\documentclass
{article}

\title{The addition theorem for locally monotileable monoid actions}
\date{}

\usepackage{amssymb,amsmath,tikz,amsthm}
\usepackage{enumerate,hyperref}
\usepackage[all]{xypic}
\usepackage{vmargin}
\setmarginsrb{10mm}{18mm}{10mm}{13mm}%
    {0mm}{0mm}{0mm}{10mm}

\usepackage{xspace}

\usepackage{centernot}

\newtheorem{lemma}{Lemma}[section]
\newtheorem{proposition}[lemma]{Proposition}
\newtheorem{theorem}[lemma]{Theorem}
\newtheorem{corollary}[lemma]{Corollary}

\newtheorem{problem}[lemma]{Problem}
\newtheorem{fact}[lemma]{Fact}
\newtheorem{claim}[lemma]{Claim}
\newtheorem{question}[lemma]{Question}

\theoremstyle{definition}
\newtheorem{definition}[lemma]{Definition}
\newtheorem{remark}[lemma]{Remark}
\newtheorem{example}[lemma]{Example}

\def\N{\mathbb N}

\def\Z{\mathbb Z}
\def\Q{\mathbb Q}

\def\P{\mathcal P}
\def\Pf{\P_{fin}}

\def\Aut{\mathrm{Aut}}

\newcommand{\halg}{h_{alg}}

\newcommand{\M}{\mathfrak M}

\newcommand*{\card}[1]{\left\vert #1 \right\vert}

\newcommand*{\set}[1]{\left\{ #1 \right\}}

\newcommand{\rest}{\mathbin\restriction}

\newcommand{\Folner}{F\o lner\xspace} 

\newcommand{\eps}{\varepsilon}

\def\tileable{locally monotileable}
\def\MT{\mathrm{MT}}
\def\MTA{\mathrm{MTA}}

\numberwithin{equation}{section}

\author{Dikran Dikranjan, Antongiulio Fornasiero, Anna Giordano Bruno, Flavio Salizzoni}
\newlength{\bibitemsep}
\newlength{\bibparskip}
\let\oldthebibliography\thebibliography
\renewcommand\thebibliography[1]{%
  \oldthebibliography{#1}%
  \setlength{\parskip}{\bibitemsep}%
  \setlength{\itemsep}{\bibparskip}%
}

\begin{document}

\maketitle

\abstract{
We prove an instance of the so-called Addition Theorem for the algebraic entropy of actions of cancellative right amenable monoids $S$ on discrete abelian groups $A$ by endomorphisms, under the hypothesis that $S$ is locally monotileable (that is, $S$ admits a right F\o lner sequence $(F_n)_{n\in\N}$ such that $F_n$ is a monotile of $F_{n+1}$ for every $n\in\N$).
We study in details the class of locally monotileable groups, also in relation with already existing notions of monotileability for groups, 
introduced by Weiss \cite{Weiss} and developed further by other authors recently.
}
\\

\noindent {\footnotesize 2010 Mathematics Subject Classification: {\sl Primary 20K30, 20M20; Secondary 37A35, 37B40, 43A07.}}

\noindent{\footnotesize Keywords: {\sl algebraic entropy, amenable semigroup, amenable monoid, group endomorphism, semigroup action, monotileable, congruent monotileable, locally monotileable.}}

\tableofcontents

\section{Introduction}

After a very brief and schematic introduction by Adler, Konheim and McAndrew \cite{AKM}, the algebraic entropy for endomorphisms of abelian groups was gradually developed by Weiss \cite{W} and Peters  \cite{P1,P2}. The interest in this direction increased after \cite{DGSZ}, where a rather complete description in the case of torsion abelian groups was obtained. 
The algebraic entropy defined by Peters  \cite{P1} for automorphisms of arbitrary abelian groups was suitably extended to arbitrary endomorphisms of abelian groups in \cite{DGB0} (see also \cite{DGB}); this entropy is denoted by $h_{alg}$ in the sequel. 
On the other hand, appropriate versions of the algebraic entropy  for module endomorphisms were introduced by Salce and Zanardo \cite{SZ1} and studied further by Salce, V\'amos and Virili \cite{SVV}, also in connection with length functions in the sense of Northcott and Reufel. 
Recently, Virili \cite{V3} extended this algebraic entropy to amenable group actions on modules and found applications to the Stable Finiteness Conjecture and the Zero Divisors Conjecture, originally stated by Kaplansky. These ideas were pushed further by Li and Liang \cite{LL}. 

Let $S$ be a cancellative right amenable semigroup, $A$ an abelian group, and $S\overset{\alpha}{\curvearrowright} A$ a left action by endomorphisms. In \cite{DFG-amac}, inspired by the recent results and definitions of Ceccherini-Silbertstein, Coornaert and Krieger \cite{CCK}, the algebraic entropy $\halg$ was extended to such actions $\alpha$ as follows. For a non-empty subset $X$ of $A$ and for every $F\in\Pf(S)$, let 
$$T_F(\alpha,X)=\sum_{s\in F}\alpha(s)(X)$$
be the \emph{$\alpha$-trajectory of $X$ with respect to $F$}. The \emph{algebraic entropy of $\alpha$ with respect to $X\in\Pf(A)$} is 
$$H_{alg}(\alpha,X)=\lim_{i\in I}\frac{\ell(T_{F_i}(\alpha,X))}{|F_i|},
$$ where $(F_i)_{i\in I}$ is a right F\o lner net of $S$. The limit defining  $H_{alg}(\alpha,X)$ exists and does not depend on the choice of the right F\o lner net $(F_i)_{i\in I}$ in view of \cite[Theorem 1.1]{CCK}. The \emph{algebraic entropy of $\alpha$} is 
$$h_{alg}(\alpha)=\sup\{H_{alg}(\alpha,X)\colon X\in \Pf(A)\}.$$

This definition of algebraic entropy coincides with that for single endomorphisms (mentioned above) when those are considered as left $\N$-actions. Moreover, for amenable group actions on discrete abelian groups it coincides with the algebraic entropy defined in \cite{V2} on locally compact abelian groups. 

\medskip 
A fundamental property of $h_{alg}$, established in \cite{DGB0} (and in \cite{DGSZ} for torsion abelian groups), is the so-called Addition Theorem (or Yuzvinski's addition formula):

\begin{theorem}[Addition Theorem]
Let $A$ be an abelian group, $\phi:A\to A$ an endomorphism and $B$ a $\phi$-invariant subgroup of $G$. Then $h_{alg}(\phi)=h_{alg}(\phi\restriction_B)+h_{alg}(\bar\phi)$, where $\bar\phi:A/B\to A/B$ is the endomorphism induced by $\phi$.
\end{theorem}

This result was generalized to locally finite groups that are either quasihamiltonian or FC-groups in \cite{GBSal}, while a counterexample to the Addition Theorem in the non-abelian case was given in \cite{GBSp}.

Moreover, an Addition Theorem was provided in \cite{SVV,SZ1} for the algebraic entropy of module endomorphisms under suitable conditions. This was extended in \cite{SV} to a more general setting, and to amenable groups actions in \cite{V3}.

In \cite{DFG-amac} we proved an Addition Theorem for actions of cancellative right amenable monoids $S$ on torsion abelian groups $A$. 
Here we prove it for all abelian groups $A$ under the hypotheses that $S$ is also countable and \tileable\ in the sense of Definition~\ref{til-def}:

\begin{theorem}[Addition Theorem]\label{ATintro}
Let $S \overset{\alpha}\curvearrowright A$ be a left action of a locally monotileable cancellative right amenable monoid $S$ on an abelian group $A$.
Let $B$ be an $\alpha$-invariant subgroup of $A$, and denote by $\alpha_{A/B}$ and $\alpha_B$ the induced actions of $S$ on $A/B$ and on $B$, respectively. Then
\[\halg(\alpha) = \halg(\alpha_{A/B})+ \halg(\alpha_B).\]
\end{theorem}

Since $\N$ is \tileable, as a corollary of Theorem~\ref{ATintro} we find the Addition Theorem from \cite{DGB0}. While the proof in \cite{DGB0} was quite long and heavily used the structure of the abelian group $A$, the proof in the present paper is much shorter and makes no recourse to the structure of $A$. 

The problem on whether the hypothesis ``locally monotileable" can be relaxed in  the Addition Theorem remains open.

\smallskip
In Section~\ref{AT-sec} we prove Theorem~\ref{ATintro} and give its consequence for the topological entropy. In particular, in \S\ref{BT-sec} we offer a background on the topological entropy of (semi)group actions and its connection to the algebraic one by means of Pontryagin duality.

\medskip
In Section~\ref{Tileable-sec} we study the class of countable \tileable\  groups.

\begin{definition}
For subsets $T,V$ of a semigroup $S$, we say that $T$ is a \emph{monotile} of $V$ if there exists a subset $C$ of $S$ such that $\{cT\colon c\in C\}$ is a partition of $V$.
\end{definition}

The notion of monotile was defined in \cite{Weiss}, in the case when $V=G$, in connection to the $\eps$-quasi tilings from \cite{OW0}. 
The interest in monotiles (of $G$) stems from the celebrated Rokhlin Lemma: 

\begin{fact}[Rokhlin Lemma]
Let $T\colon X\to X$ be an invertible measure-preserving transformation on a probability space $(X,\Sigma,\mu)$. We assume $T$ is (measurably) aperiodic, that is, the set of periodic points for $T$ has zero measure. Then for every integer $n\in\mathbb{{N}}_+$ and for every $\varepsilon >0$, there exists a measurable set $E$ such that the sets $E$, $TE$, $\ldots$, $T^{n-1}E$ are pairwise disjoint and such that $\mu (E\cup TE\cup \cdots \cup T^{n-1}E)>1-\varepsilon$.
\end{fact}

An extension of Rokhlin Lemma for $\Z^d$-actions was proved in \cite{Con} and in \cite{KW}.
A further extension of this result for amenable group actions was announced in \cite{OW0} and then proved in \cite{OW}. More precisely, 
if a countable amenable group $G$ acts freely on a Lebesgue measure space $(X,\mu)$, we say that Rokhlin Lemma holds for a finite subset $F$ of $G$ if for every $\varepsilon>0$ there is a subset $B$ of $X$ such that the sets in $\{fB\colon f\in F \}$ are pairwise disjoint and $\mu(\bigcup_{f\in F}fB)>1-\varepsilon$. Rokhlin Lemma holds for $F$ if and only if $F$ is a monotile of $G$ (see \cite{OW}).

\medskip
For our purpose concerning the Addition Theorem, we need the following special right F\o lner sequences.

\begin{definition}\label{til-def}
Let $S$ be a monoid. A sequence $(F_n)_{n \in \N}$ in $\Pf(S)$ is \emph{locally monotileable} if $F_0=\{1\}$ and $F_n$ is a monotile of $F_{n+1}$ for every $n \in \N$.

A countable right amenable monoid is \emph{locally monotileable} if it admits a locally monotileable right F\o lner sequence.
\end{definition}

Consider a \tileable\  sequence $(F_n)_{n\in\N}$ of $S$. By definition for every $n\in\N_+$ there is a finite subset $K_n$ of $F_n$ such that
$F_n=\bigsqcup_{k\in K_n}kF_{n-1}$ (in particular, $F_n=K_nF_{n-1}$).
Let $K_0=\{1\}$. Since $F_0=\{1\}$, we have that $K_1=F_1$ and so by induction we conclude that, for every $n\in\N$, 
$$F_n=K_n\dots K_1 K_0.$$
The sequence $(K_n)_{n\in\N}$ is the \emph{tiling sequence} associated to the \tileable\  sequence $(F_n)_{n\in\N}$.


\smallskip
Definition~\ref{til-def} is inspired by a notion due to Weiss, that he introduced for groups in \cite{Weiss}:

\begin{definition}
Let $S$ be a monoid.
A countable cancellative right amenable monoid $S$ is \emph{monotileable amenable} (briefly, $\MTA$) if there exists a right F\o lner sequence $(F_n)_{n\in\N}$ of $S$ such that $F_n$ is a monotile of $S$ for every $n\in\N$.
\end{definition}

The following special case of monotileability was introduced in \cite[Definition 4]{CC} in the case of groups, we now give it using our terminology. 

\begin{definition}\label{CCdef}
Let $S$ be a monoid. A sequence $(F_n)_{n\in\N}$ in $\Pf(S)$ is \emph{congruent} if it is locally monotileable and it admits a tiling sequence $(K_n)_{n\in\N}$ with $1\in K_n$ for every $n\in\N$. Moreover, $(F_n)_{n\in\N}$ is \emph{exhaustive} if $\bigcup_{n\in\N}F_n=S$.

A countable right amenable monoid $S$ is \emph{congruent monotileable} if it admits an exhaustive congruent right F\o lner sequence.
\end{definition}

Note that, if $S$ is a cancellative monoid and $(F_n)_{n\in\N}$ is a locally monotileable sequence of $S$ with associated tiling sequence $(K_n)_{n\in\N}$ and $1\in K_n$ for every $n\in\N$, then $(F_n)_{n\in\N}$ is increasing, that is, $F_n\subseteq F_{n+1}$ for every $n\in\N$, and moreover $K_n\subseteq F_{n+1}$ for every $n\in\N$.  Example~\ref{0notin} shows that the converse is not true in general.

\medskip 
For a right F\o lner sequence $(F_n)_{n\in\N}$ of a monoid $S$, the following (non-)implications hold.
$$
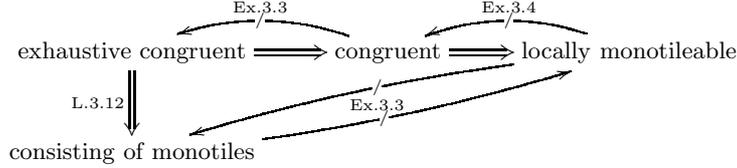
\xymatrix{
\text{exhaustive congruent} \ar@{=>}[r] \ar@{=>}[d]_{\text{L.\ref{claim*}}} &  \text{congruent} \ar@/_1pc/@{->}[l]|{/}_{\text{Ex.\ref{esemono}}} \ar@{=>}[r] & \text{locally monotileable} \ar@/_1pc/@{->}[l]|{/}_{\text{Ex.\ref{0notin}}} \ar@/_0.5pc/@{->}[lld]|{/}\\
\text{consisting of monotiles} \ar@/_0.5pc/@{->}[rru]|{/}^{\text{Ex.\ref{esemono}}}  & &
}$$

Since an exhaustive congruent sequence of a monoid $S$ consists necessarily of monotiles of $S$ (see Lemma~\ref{claim*}), a countable congruent monotileable monoid is $\MTA$. In particular, the condition of monotileability in \cite[Definition 4]{CC} is redundant. 

For countable right amenable monoids, one has the following implications among the notions of monotileability introduced and recalled above. 
A counterexample witnessing that locally monotileable does not imply $\MTA$ (and so neither congruent monotileble) is given in Example~\ref{New:Example}.

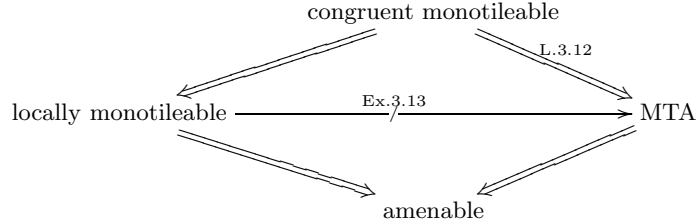
$$\xymatrix{& \text{congruent monotileable}\ar@{=>}[dr]^{\text{L.\ref{claim*}}} \ar@{=>}[dl] & \\
\text{\tileable}\ar@{=>}[dr] \ar@{->}[rr]|{/}^{\text{Ex.\ref{New:Example}}} & & \MTA\ar@{=>}[dl] \\
& \text{amenable} &
}$$

\medskip
Restricting to the case of groups, first note that for groups the notions of local monotileability and congruent monotileability are equivalent (see Proposition~\ref{lm=cm}).

Moreover, recall that Weiss \cite{Weiss} proved that every countable residually finite amenable group is $\MTA$ and that every countable solvable group is $\MTA$. The latter result was extended by Ebli \cite{Ebli}, showing that every countable elementary amenable group is $\MTA$. So, the next related question is very natural.

\begin{question}\label{Ebli-Q}
Is every $\MTA$ group necessarily elementary amenable?
\end{question}

The following very general question by Weiss is open.

\begin{question}[See \cite{Weiss}]\label{Q-Weiss}
Is every countable amenable group necessarily $\MTA$?
\end{question}

In this sense Downarowicz, Huczek and Zhang \cite{DHZ} proved that every countable amenable group is \emph{finitileable}. This term was coined later by Danilenko \cite{Dan}, who gave also a shorter dynamical proof of the same result, as well as a version of Rokhlin Lemma for finitileable groups.

Inspired by the notion from \cite{Dan,DHZ}, Cecchi and Cortez \cite{CC} introduced the notion recalled above of congruent monotileable group. Cortez and Petite proved in \cite{CP} that residually finite amenable groups are congruent monotileable. Using this result, Cecchi and Cortez showed that every countable virtually nilpotent group is congruent monotileable (see \cite[Theorem 1]{CC}). The following questions from \cite{CC}, connected to the general Question~\ref{Q-Weiss}, are open.

\begin{question}[See \cite{CC}]\label{Q-CC}
\begin{enumerate}[(a)]
\item Is every countable amenable group necessarily congruent monotileable?
\item In particular, is every countable $\MTA$ group necessarily congruent monotileable?
\end{enumerate}
\end{question}

In \cite{Weiss}, also the following notion was introduced in the case of groups. 

\begin{definition}
A semigroup $G$ is \emph{monotileable} (briefly, $\MT$) if every finite subset of $S$ is contained in a finite monotile of $S$. 
\end{definition}

While every $\MTA$ group is necessarily $\MT$ and amenable (see Proposition~\ref{MTA->MT}), the validity of the converse implication is not known, and the following question by Weiss is open.

\begin{question}[See \cite{Weiss}] \label{WeissQues}
If a countable group $G$ is $\MT$ and amenable, is $G$ necessarily $\MTA$?
\end{question}

We are not aware if a negative answer of the counterpart of Question~\ref{WeissQues} for cancellative monoids is available. 

\medskip
In order to study the class $\M$ of locally monotileable groups, we consider the following general problem concerning the stability of $\M$ under extension. 

\begin{problem}\label{ticon1}
Consider three countable groups $G$, $H$ and $K$, such that 
$0 \rightarrow H \xrightarrow{\iota} G \xrightarrow{\pi} K \rightarrow 0$
is a short exact sequence of groups. 
\begin{enumerate}[(a)]
\item If $H$ and $K$ are \tileable, is then $G$ \tileable \ as well?
\item What about splitting extensions $ G = H \rtimes K$?
\end{enumerate}
\end{problem}

Moreover, we introduce the following notion of monotileability stronger than local monotileability.

\begin{definition}
Let $G$ be a group, $(F_n)_{n\in\N}$ a sequence of $\Pf(G)$ and $\mathrm{id}\in X \subseteq \Aut(G)$. We say that $(F_n)_{n\in\N}$ is a $X$-\emph{monotileable} sequence if for all $n\in\N$ and $\phi\in X$, we have that $\phi(F_n)$ is a monotile of $F_{n+1}$.
We say that $G$ is $X$-\emph{monotileable} if there exists a $X$-monotileable right F\o lner sequence $(F_n)_{n\in\N}$ of $G$.

When $X=\{\mathrm{id},\psi,\psi^{-1}\}$, we simply write $\psi$-\emph{monotileable.}
\end{definition}

One of our main results is the following partial answer to Problem~\ref{ticon1}.

\begin{theorem}[Extension Theorem]\label{monosequ1}
Consider three countable groups $G$, $K$ and $H$. Suppose that 
\begin{equation*}
0 \rightarrow H \xrightarrow{\iota} G \xrightarrow{\pi} K \rightarrow 0
\end{equation*}
is a short exact sequence. If $K$ is \tileable\ and $H$ is $\mathrm{Inn}(G)$-monotileable, then $G$ is \tileable.
\end{theorem}

Using this result,  we prove one of the main achievements of this paper, that is, Theorem~\ref{vhy}. Recall that a group $G$ is \emph{hypercentral} if its upper central series terminates at the whole group, that is, there exists an ordinal $\alpha$ such that $Z_\alpha(G)=G$; the length of $G$ as an hypercentral group is the minimum such $\alpha$. 

\begin{theorem}\label{vhy}
Every countable virtually hypercentral group of length $<\omega^2$ is locally monotileable (i.e., congruent monotileable). 	
\end{theorem}

Clearly every nilpotent group is hypercentral. So, as an immediate corollary of Theorem~\ref{vhy} we obtain the above mentioned result from \cite{CC}, that every countable virtually nilpotent group is congruent monotileable. 

\smallskip
Since all countable solvable groups are known to be $\MTA$, and in view of Theorem~\ref{vhy}, it is natural to ask the following.

\begin{question}\label{question}
Are all countable solvable groups \tileable? What about polycyclic groups?
\end{question}

A positive answer to Problem~\ref{ticon1} would also be a positive answer to Question~\ref{question}.  In this context, the next theorem provides an example of a \tileable\ solvable (actually, metabelian) group that is neither virtually nilpotent nor residually finite.

\begin{theorem}\label{tipro3}
For every automorphism $\phi$ of $\Q$, the group $\Q$ and $\Q \rtimes_{\phi}\Z$ are \tileable.	
\end{theorem}

The following diagram represents all known (non-)implications among the above mentioned properties for countable groups, and the related open questions.

$$
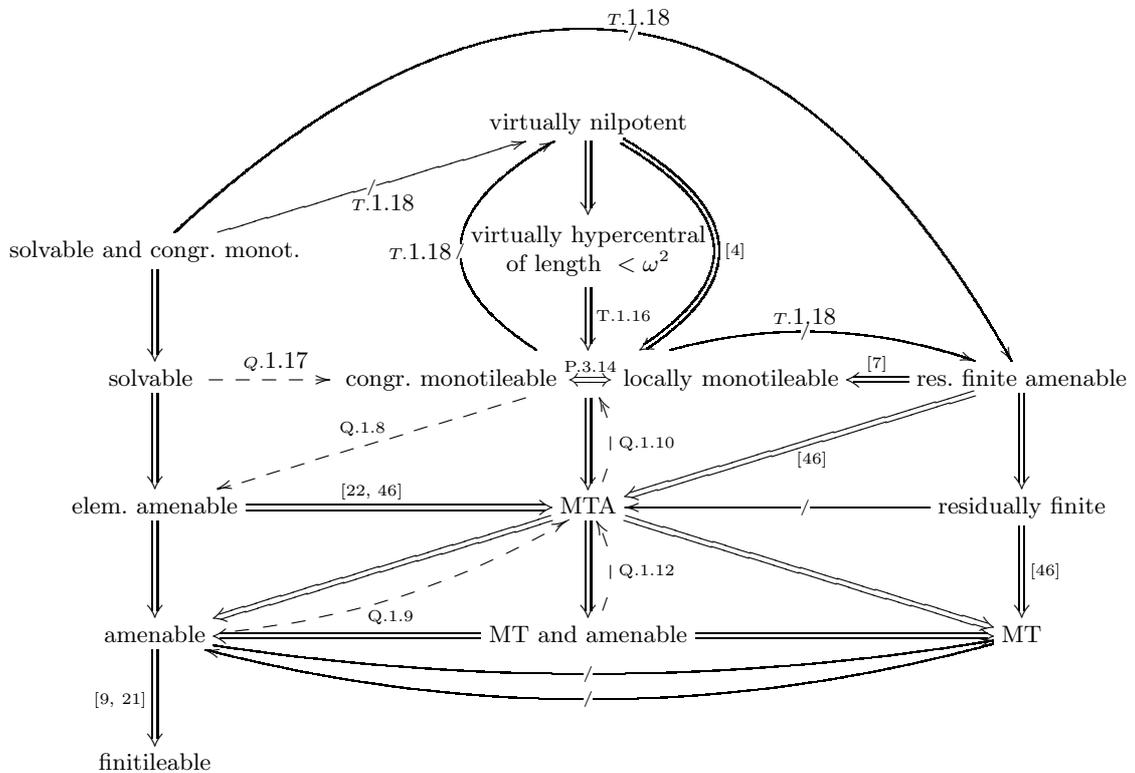
\xymatrix@!0@C5.7cm@R=1.7cm{
& \text{virtually nilpotent}  \ar@{=>}[d] \ar@/^4pc/@{=>}[dd]^{\text{\cite{CC}}}&\\
\text{solvable and congr.~monot.}\ar@{=>}[d]\ar@{->}[ru]|{/}_{T.\ref{tipro3}}\ar@/^9pc/@{->}[rrd]|{/}^{T.\ref{tipro3}}& {\begin{array}{@{}c@{}} \text{virtually hypercentral} \\ \text{of length}\ <\omega^2 \end{array}}\ar@{=>}[d]^{\text{T.\ref{vhy}}}&\\
\text{solvable}\ \ar@{-->}[r]^{\!\!\!\!\!\!\!\!\!\!\!\!\!\!\!\!\!\!\!\!\!\!\!\!\!\!\!\!\!\!\!\!\!\!\!\!\!\!\!\!\!\!\!\!\!\!\!\!\!\!\!\!\!\!\!\! Q.\ref{question}} \ar@{=>}[d] & \ \text{congr.~monotileable}\overset{\text{P.\ref{lm=cm}}}{\Longleftrightarrow} \text{locally monotileable}\ar@{=>}[d] \ar@{-->}[dl]_{\text{Q.\ref{Ebli-Q}}}\ar@/^1.5pc/@{->}[r]|{/}^{T.\ref{tipro3}}\ar@/^4pc/@{->}[uu]|{/}^{T.\ref{tipro3}}& \text{res.~finite amenable}\ar@{=>}[l]_{\text{\cite{CP}}\!\!\!\!\!\!\!\!\!\!\!\!\!\!\!\!\!\!\!\!\!\!\!\!\!\!\!\!\!\!\!\!\!\!\!\!\!\!\!\!\!\!} \ar@{=>}[d]\ar@{=>}[dl]^{\text{\cite{Weiss}}}\\
\text{elem.~amenable}\ar@{=>}[r]^{\text{\cite{Ebli,Weiss}}}\ar@{=>}[d] & \MTA \ar@{=>}[d] \ar@/_0.7pc/@{-->}[u]_{\text{Q.\ref{Q-CC}}}\ar@{=>}[dl]\ar@{=>}[dr] & \text{residually~finite}\ar@{=>}[d]^{\text{\cite{Weiss}}}\ar@{->}[l]|{/} \\
\text{amenable}\ar@{=>}[d]_{\text{\cite{Dan,DHZ}}}\ar@/_1.2pc/@{->}[rr]|{/} \ar@/_1pc/@{-->}[ur]_{\text{Q.\ref{Q-Weiss}}}& \MT\ \text{and amenable}\ar@/_0.7pc/@{-->}[u]_{\text{Q.\ref{WeissQues}}}\ar@{=>}[l]\ar@{=>}[r] & \MT \ar@/^2pc/@{->}[ll]|{/} \\
\text{finitileable} & &
}$$

The following open questions, related to Problem~\ref{ticon1} and Question~\ref{question}, are motivated by Theorem~\ref{tipro3}.

\begin{question}
\begin{enumerate}[(a)]
\item Is the group $\Q^n \rtimes_{\phi}\Z$ \tileable \ for every automorphism $\phi$ of $\Q^n$?
\item More generally, is the group $H \rtimes_{\phi}\Z$ \tileable \ for an abelian group $H$ and $\phi\in\Aut(H)$?
\item If $H$ is a \tileable\ group and $K$ is a finitely generated \tileable\ subgroup of $\Aut(H)$, is the group $H\rtimes K$ \tileable?
\end{enumerate}
\end{question}

Call a group $G$ \emph{hereditarily locally monotileable} (briefly, \emph{h-\tileable}) if every countable subgroup of $G$ is \tileable.
It can be easily deduced from Theorem~\ref{vhy} that every virtually nilpotent group, as well as every locally nilpotent group, is h-\tileable. 

\begin{problem}\label{ticon2}
Find examples of finitely generated amenable groups that are not \tileable, or at least h-\tileable.
\end{problem}

In \S\ref{ex-sec} we provide further examples, showing an example of a locally monotileable group that is neither virtually solvable nor locally finite, nor residually solvable, and of a locally monotileable group $G$ that is virtually hypercentral, yet neither virtually nilpotent nor residually finite.

\smallskip
We show that $\mathfrak M$ is stable under some countable direct limits, in particular, under countable direct sums (so, under finite direct products as well). %
So, we conclude with the following open question about basic stability properties of $\M$.

\begin{question}
Is $\M$ stable under taking subgroups or quotients?
\end{question}

\subsubsection*{Acknowledgments}

It is a pleasure to thank Benjamin Weiss for the fruitful discussions with the second named author and for the wealth of useful suggestions, also concerning the papers \cite{DFG-amac,DFG}. In particular his suggestion to work with monotiles lead to this work.

\smallskip
This work was partially supported by the ``National Group for Algebraic and Geometric Structures, and their Applications'' (GNSAGA - INdAM).

\subsubsection*{Notation and terminology}

For a set $X$, we let $\ell(X)=\log|X|,$ using the convention that $\ell(X)=\infty$ if the set $X$ is infinite. Moreover, we denote by $\P(X)$ the family of all subsets of $X$ and by $\Pf(X)$ its subfamily consisting of all non-empty finite subsets of $X$.
For an abelian group $A$, let $\Pf^0(A)$ be the family of all finite subsets of $A$ containing $0$.

In case $S$ is a monoid with neutral element $1$, a left semigroup action $S\overset{\alpha}{\curvearrowright}X$ is a \emph{left monoid action} of $S$ on $X$ if $\alpha(1)(x) = x$ for all $x\in X$, i.e., $\alpha(1)$ is the identity map $id_X$. If $S$ is a group, then this condition implies that $\alpha(s)$ is a bijection for every $s\in S$.
Unless otherwise stated, all the actions of monoids considered in this paper are monoid actions.

We recall that a  \emph{right F\o lner net} of a semigroup $S$ is a net $(F_i)_{i\in I}$ in $\Pf(S)$ such that
$\lim_{i\in I}{|F_is\setminus F_i|}/{|F_i|}=0$ for every $s\in S$.
By \cite[Corollary 4.3]{N}, a semigroup $S$ is \emph{right amenable} if and only if $S$ admits a right F\o lner net.
In case $S$ is commutative we omit the adjective ``right".

\section{Addition Theorem}\label{AT-sec}

%
%
%

\subsection{Properties of the trajectories}

%

%

\begin{lemma}\label{disjoint}
Let $S$ be a semigroup, $A$ an abelian group, and $S\overset{\alpha}\curvearrowright A$ a left action.
If $F,F'\in\Pf(S)$ are disjoint and $X\in\Pf^0(A)$, then $T_{F\sqcup F'}(\alpha,X)=T_F(\alpha,X)+T_{F'}(\alpha,X).$
\end{lemma}


Note that the inclusion $T_{FF'}(\alpha,X)\subseteq T_F(\alpha,T_{F'}(\alpha,X))$ in the next lemma holds in general.

\begin{lemma}\label{TEF}
Let $S$ be a right amenable monoid, $A$ an abelian group, $S\overset{\alpha}\curvearrowright A$ a left action. If $F,F'\in\Pf(S)$ contain $1$, $X\in\Pf^0(A)$ and the sets $\{fF'\colon f\in F \}$ are pairwise disjoint, then $T_{FF'}(\alpha,X)=T_F(\alpha,T_{F'}(\alpha,X)).$
\end{lemma}

\begin{proof} By Lemma~\ref{disjoint}, we have that
\begin{equation*}T_{FF'}(\alpha,X)=T_{\bigsqcup_{f\in F}fF'}(\alpha,X)=\sum_{f\in F}T_{fF'}(\alpha,X)=\sum_{f\in F}\alpha(f)(T_{F'}(\alpha,X))=T_F(\alpha,T_{F'}(\alpha,X)).\qedhere\end{equation*}
\end{proof}

\begin{lemma}\label{D5}
Let $S$ be a right amenable semigroup, $A$ an abelian group, $S\overset{\alpha}\curvearrowright A$ a left action, $F\in\Pf(S)$ and $X,Y\in\Pf^0(A)$.  Then
$T_F(\alpha,X+Y)=T_F(\alpha,X)+T_F(\alpha,Y).$
\end{lemma}
\begin{proof}
By definition 
\begin{align*}T_F(\alpha,X+Y)&=\sum_{s\in F}\alpha(s)(X+Y)=\sum_{s\in F}(\alpha(s)(X)+\alpha(s))(Y))=\\&=\sum_{s\in F}\alpha(s)(X)+\sum_{s\in S}\alpha(s)(Y)=T_F(\alpha,X)+T_F(\alpha,Y).\qedhere\end{align*}
\end{proof}

The proof of the next lemma is straightforward.

\begin{lemma}\label{pi}
Let $S$ be a right amenable semigroup, $A$ an abelian group, $S\overset{\alpha}\curvearrowright A$ a left action, and $B$ an $\alpha$-invariant subgroup of $A$ with $\pi:A\to A/B$ the canonical projection. If $X\in\Pf^0(A)$, $F\in \Pf(S)$, then $T_F(\alpha_{A/B},\pi(X))=\pi(T_F(\alpha,X)).$
\end{lemma}

\subsection{The function $\ell(-,-)$}

Let $A$ be an abelian group. For $X,Y\in\P(A)$ let
\begin{equation}\label{Def:mu}
\mu(X,Y)=\min\left\{n\in\N\colon \exists a_0=0,a_1,\ldots,a_{n-1}\in A,\ X\subseteq \bigcup_{i=0}^{n-1}(a_i+Y)\right\}
\end{equation}
and $\ell(X,Y)=\log \mu(X,Y).$
If $X$ and $Y$ are subgroups of $A$, then $\mu(X,Y)=\mu(X+Y,Y)=[X+Y\colon Y]$; in particular, if $Y\leq X$, then $\mu(X,Y)=[X\colon Y]$.

\medskip
Obviously, the family $$\mathfrak Y= \{a_0+Y, a_1+Y, \ldots , a_{n-1}+Y\}$$ appearing in \eqref{Def:mu} is pairwise disjoint when $(C-C)\cap (Y-Y) = \{0\}$, where $C=\{a_0, a_1, \ldots , a_{n-1}\}$. We say that $\mathfrak Y$ is a {\em strongly pairwise disjoint}, if the ``fattened" family $$\mathfrak Y^*= \{a_0+Y-Y, a_1+Y-Y, \ldots , a_{n-1}+Y-Y\}$$ is still pairwise disjoint, or equivalently, when $(C-C)\cap (Y-Y+Y-Y) = \{0\}$. 

\medskip
In the following lemma we collect other useful properties of the function $\ell(-,-)$. 

\begin{lemma}\label{lem:ell}
Let $A$ an abelian group and $X,Y,Z,X',Y'\in\Pf^0(A)$. Then:
\begin{enumerate}[(a)]
\item\label{ell0} $\ell(X,Y)\geq 0$, $\ell(X,X)=0$ and $\ell(X)=\ell(X,\{0\})$;
\item\label{ell1} the function $\ell(X, Y)$ is increasing in $X$ and decreasing in $Y$;  
\item\label{ell2} $\ell(X,Y)\leq  \ell(X) \leq \ell(X,Y)+\ell(Y)$;
\item\label{ell3}  $\ell(X + X' , Y + Y') \leq \ell(X ,Y) + \ell(X', Y')$;
\item\label{ell4} $\ell(X,Y)\leq \ell(X,Z)+\ell(Z,Y)$;
\item\label{ell5} if $\varphi:A\to A$ is an endomorphism, then $\ell(\varphi(X),\varphi(Y))\leq\ell(X,Y)$;
\item\label{ell6} if $a_0=0,a_1,\ldots,a_{n-1}\in A$ are such that $X\subseteq \bigcup_{i=0}^{n-1}(a_i+Y)$, the family $\mathfrak Y= \{a_0+Y, a_1+Y, \ldots , a_{n-1}+Y\}$ is strongly pairwise disjoint and $X$ meets $a_i+Y$ for every $i\in\{0,1,\ldots,n\}$, then $\mu(X,Y)=n$.
\end{enumerate}
\end{lemma}
\begin{proof} 
\eqref{ell0}, \eqref{ell1} and \eqref{ell2} are clear.

\eqref{ell3} Follows from the fact that if $a_0=0,a_1,\ldots,a_{n-1}\in A$ and $a'_0=0,a'_1,\ldots,a'_{m-1}\in A$ are such that
$$X\subseteq (a_0+Y)\cup \ldots\cup (a_{n-1}+Y) \mbox{ and }X'\subseteq (a'_0+Y')\cup\ldots\cup (a'_{m-1}+Y'),$$
then $$X+X'\subseteq (a_0+Y+Y')\cup\ldots\cup (a_{n-1}+Y+Y')\cup (a'_0+Y+Y')\cup\ldots\cup (a'_{m-1}+Y+Y').$$

\eqref{ell4} Similarly, if $a_0=0,a_1,\ldots,a_{n-1}\in A$ and $b_0=0,b_1,\ldots,b_{m-1}\in A$ are such that 
$$X\subseteq (a_0+Z)\cup \ldots\cup (a_{n-1}+Z)\mbox{ and }Z\subseteq (b_0+Y)\cup\ldots\cup (b_{m-1}+Y),$$ 
then $$X\subseteq \bigcup_{i\in\{0,\ldots,n-1\},j\in\{0,\ldots,m-1\}}(a_i+b_j+Y).$$

\eqref{ell5} Let $a_0=0,a_1,\ldots,a_{n-1}\in A$ such that $X\subseteq (a_0+Y)\cup (a_1+Y)\cup\ldots\cup (a_{n-1}+Y).$
Then 
\[\begin{split}\varphi(X)&\subseteq\varphi((a_0+Y)\cup (a_1+Y)\cup\ldots\cup (a_{n-1}+Y))\\&=(\varphi(a_0)+\varphi(Y))\cup (\varphi(a_1)+\varphi(Y))\cup\ldots\cup(\varphi(a_{n-1})+\varphi(Y)).\end{split}\]

\eqref{ell6} Obviously, $\mu(X,Y)\leq n$. Assume that $m:=\mu(X,Y) < n$. Then there exist $b_0=0,b_1, \ldots, b_{m-1}\in A$ such that 
\begin{equation}\label{Eq1}
X\subseteq \bigcup_{j=0}^{m-1}(b_j+Y).
\end{equation}
For every $i\in\{0,1,\ldots n-1\}$ there exists $x_i\in X\cap (a_i + Y)$, by hypothesis.
From \eqref{Eq1} and our assumption, $m<n$ implies that there exist $0\leq i < j < n$ and $k\in\{1,\ldots,m\}$ such that $x_i, x_j\in b_k + Y$. Hence, $x_i - x_j \in  Y-Y$. Since moreover $x_i-x_j\in a_i-a_j+Y-Y$, it follows that $a_i-a_j\in Y-Y+Y-Y$. For $C=\{a_0, a_1, \ldots , a_{n-1}\}$, our hypothesis on $\mathfrak Y$ gives $a_i-a_j\in(C-C)\cap (Y-Y+Y-Y) = \{0\}$, and so $a_i-a_j=0$, a contradiction.
\end{proof}

\begin{proposition}\label{D7}\label{D8}
Let $A$ be an abelian group, $B$ a subgroup of $A$ and $\pi:A\to A/B$ the canonical projection. If $X\in\Pf^0(A)$, then:
\begin{enumerate}[(a)]
\item there exists $Y\in\Pf^0(B)$ such that $\ell(\pi(X))=\ell(X,B) = \ell(X,Y)$;
\item for $Y\in\Pf^0(B)$ we have that $\ell(X+Y)\geq \ell(\pi(X))+\ell(Y)$.
\end{enumerate}
\end{proposition}

\begin{proof} (a) The equality 
\begin{equation}\label{AF}
\ell(\pi(X))=\ell(X,B)
\end{equation}
is obvious. 

Let $Y=(X-X)\cap B\in\Pf^0(B)$ and let $Z=\{a_1,\ldots,a_{n-1}\}\subseteq X\setminus B$ be such that $\pi(X) \subseteq \pi(Z) \cup\{0\}$ and 
$(Z-Z)\cap B=\{0\}$,  in other words $Z\cup\{0\}$ is a set of representatives of $\pi(X)$. Then 
$$
\card{Z\cup\{0\}}=\card{\pi(Z\cup\{0\})}=\card{\pi(X)}.
$$ 
In particular, putting $a_0=0$, one obtains
\begin{equation}\label{subseteq}
X\subseteq (a_0+Y)\cup(a_1+Y)\cup\ldots\cup(a_{n-1}+Y).
\end{equation}
In fact, assume that $x\in X$. If $x\in B$, so $x\in X\cap B\subseteq Y$. Otherwise, if $x\not\in B$, there exists $i\in\{1,\ldots,n-1\}$ such that $x\in a_i+B$; since $x-a_i\in (X-X)\cap B=Y$, one concludes that $x\in a_i+Y$.

On the other hand, by the choice of $Z$, the family $\{a_i+Y\colon i\in\{0,\ldots,n-1\}\}$ appearing in \eqref{subseteq} is strongly pairwise disjoint. Indeed, $Y-Y+Y-Y \subseteq B$, as $Y\subseteq B$. Therefore, $$(Z-Z)\cap (Y-Y+Y-Y)\subseteq (Z-Z)\cap B=\{0\}.$$ 
By Lemma~\ref{lem:ell}\eqref{ell6}, this shows also that $\mu(X,Y)=n = \pi(X)$. Therefore, $\ell(\pi(X))=\ell(X,Y)$. 

%
%
%
%

\medskip
(b) Let $Z\in\Pf^0(A)$ such that $Z\subseteq X$, $\pi(Z)=\pi(X)$ and $(Z-Z)\cap B=\{0\}$; then $\card{Z}=\card{\pi(Z)}=\card{\pi(X)}.$
The bijectivity of the map $Z\times Y\to Z+Y,\ (z,y)\mapsto z+y$ entails 
\begin{equation}\label{Dic:19}
\card{Z+Y}=\card{Z}\card{Y}.
\end{equation}
The inclusion $Z+Y\subseteq X+Y$ and (\ref{Dic:19}) yield the required inequality \[\ell(X+Y)\geq \ell(Z+Y)=\ell(Z)+\ell(Y)=\ell(\pi(X))+\ell(Y).\qedhere\]
\end{proof}

\begin{lemma}\label{D9}
Let $S$ be a right amenable semigroup, $A$ an abelian group, $S\overset{\alpha}\curvearrowright A$ a left action. If $F\in\Pf(S)$ and $X,Y\in\Pf^0(A)$, then
$\ell(T_F(\alpha,X),T_F(\alpha,Y))\leq |F|\ell(X,Y).$
\end{lemma}
\begin{proof}
By definition, Lemma~\ref{lem:ell}\eqref{ell3} and Lemma~\ref{lem:ell}\eqref{ell5}, we have that
\[\begin{split}\ell(T_F(\alpha,X),T_F(\alpha,Y))&=\ell\left(\sum_{s\in F}\alpha(s)(X),\sum_{s\in F}\alpha(s)(Y)\right)\leq \\ &\leq\sum_{s\in F}\ell(\alpha(s)(X),\alpha(s)(Y))\leq \sum_{s\in F}\ell(X,Y)=\card{F}\ell(X,Y)\qedhere.\end{split}\]
\end{proof}

\subsection{Proof of the Addition Theorem}

We recall the following basic property of the algebraic entropy that is used in the sequel without referring to it each time.

\begin{lemma}[See \cite{DFG-amac}]
Let $S$ be a cancellative right amenable semigroup, $A$ an abelian group, and $S\overset{\alpha}{\curvearrowright}A$ a left action.
If $X,Y\in \Pf(A)$ and $X\subseteq Y$, then $H_{alg}(\alpha,X)\leq H_{alg}(\alpha,Y)$.
Consequently, if $\mathcal F\subseteq\Pf(A)$ is cofinal with respect to $\subseteq$, then $h_{alg}(\alpha)=\sup\{H_{alg}(\alpha,X): X\in\mathcal F\}.$
\end{lemma}


By the above lemma, it is clear that
$$h_{alg}(\alpha)=\sup\{H_{alg}(\alpha,X)\colon X\in\Pf^0(A)\}.$$

The following is the key point for the proof of the Addition Theorem.

\begin{proposition}\label{D10}
Let $S$ be a countable \tileable\  cancellative right amenable monoid, and let $(F_n)_{n \in \N}$ be a \tileable\  right \Folner sequence of $S$. Let $A$ be an abelian group and $S\overset{\alpha}\curvearrowright A$ a left action. Let $X,Y\in\Pf^0(A)$. Then the following functions are decreasing:
\[\N\ni n  \mapsto \frac{\ell(T_{F_n}(\alpha,X))}{\card{F_n}}\quad \text{and}\quad \N \ni n \mapsto \frac{\ell(T_{F_n}(\alpha,X) , T_{F_n}(\alpha,Y))}{\card{F_n}}.\] 
\end{proposition}
\begin{proof}
The first assertion follows from the second one by taking $Y= \set 0$. To prove the second assertion
let $n \in \N$. Then there exists $K=K_{n+1}$ such that $F_{n+1} = \bigsqcup_{s\in K} s F_n$; in particular, $\card{F_{n+1}}=\card{K}\card{F_n}$.
Then by Lemma~\ref{TEF}, $$T_{F_{n+1}}(\alpha,X)=T_{KF_n}(\alpha,X)=T_K(\alpha,T_{F_n}(\alpha,X)).$$
The same holds for $Y$.
Therefore, by Lemma~\ref{D9},
\begin{equation*}
\ell(T_{F_{n+1}}(\alpha,X) , T_{F_{n+1}}(\alpha,Y))=\ell(T_K(\alpha,T_{F_n}(\alpha,X)),T_K(\alpha,T_{F_n}(\alpha,Y)))\leq \card{K}\ell(T_{F_{n}}(\alpha,X) , T_{F_{n}}(\alpha, Y)),
\end{equation*}
and so
\[\begin{split}\frac{\ell(T_{F_{n+1}}(\alpha,X) , T_{F_{n+1}}(\alpha,Y))}{\card{F_{n+1}}} \leq
\frac{\card{K}\ell(T_{F_{n}}(\alpha,X) , T_{F_{n}}(\alpha, Y))}{\card{K}\card{F_{n}}}=\frac{\ell(T_{F_{n}}(\alpha,X) , T_{F_{n}}(\alpha, Y))}{\card{F_{n}}};\end{split}\]
this proves the second assertion. 
 \end{proof}

We are now in position to prove the Addition Theorem.

\begin{proof}[\bf Proof of Theorem~\ref{ATintro}] We have to prove that if  $S$ be a countable cancellative right amenable monoid
with a \tileable\  right F\o lner sequence $(F_n)_{n\in\N}$, $A$ is an abelian group, $S\overset{\alpha}{\curvearrowright}A$ a left action, 
and $B$ is an $\alpha$-invariant subgroup of $A$, then $h_{alg}(\alpha) = h_{alg}(\alpha_{A/B})+h_{alg}(\alpha_B).$

First we prove the inequality
\begin{equation}\label{geq}
h_{alg}(\alpha) \geq \halg(\alpha_{A/B})+\halg(\alpha_B).
\end{equation}
Let $\pi:A\to A/B$ be the canonical projection, $X\in\Pf^0(A)$ and $Y\in\Pf^0(B)$. Pick $Z\in\Pf^0(A)$, as in the proof of Proposition~\ref{D8}(b), i.e., with $\pi(Z)=\pi(X)$ and $(Z-Z) \cap B = \{0\}$, so that $\card{Z}=\card{\pi(Z))}=\card{\pi(X)}$.
For every $n\in\N$, by Lemma~\ref{D5}, Proposition~\ref{D8}(b) and Lemma~\ref{pi}, one has the inequalities
\[\begin{split}\ell(T_{F_n}(\alpha,Z+Y))&=\ell(T_{F_n}(\alpha,Z)+T_{F_n}(\alpha,Y))\geq \\ &\geq \ell(\pi(T_{F_n}(\alpha,Z))) +\ell(T_{F_n}(\alpha,Y))=\ell(T_{F_n}(\alpha_{A/B},\pi(Z)))+\ell(T_{F_n}(\alpha,Y)).\end{split}\]
Since $\pi(X)=\pi(Z)$, after division by $|F_n|$ and letting $n\to \infty$ 
one obtains the inequalities
$$\halg(\alpha)\geq H_{alg}(\alpha,Z+Y)\geq H_{alg}(\alpha_{A/B},\pi(X))+H_{alg}(\alpha_B,Y).$$
So, taking the supremum on the right-hand side of the latter inequality over all $X\in\Pf^0(A)$ and $Y\in\Pf^0(B)$, gives the inequality in \eqref{geq}.

It remains to prove  the inequality
\begin{equation}\label{leq}
h_{alg}(\alpha) \leq  \halg(\alpha_{A/B})+\halg(\alpha_B).
\end{equation}
Fix $\eps>0$ and $X\in\Pf^0(A)$. By Lemma~\ref{D10}, there exists $M\in\N$ such that, for every $n \geq M$,
\begin{equation}\label{-3-}
\frac{\ell(T_{F_n}(\alpha_{A/B}, \pi(X)))}{\card{F_n}} \leq H_{alg}(\alpha_{A/B}, \pi(X)) + \eps\leq \halg(\alpha_{A/B})+\eps. 
\end{equation}
By Proposition~\ref{D7}(a), there exists $Y\in\Pf^0(B)$ such that  $\ell(\pi(T_{F_M}(\alpha,X)))=\ell(T_{F_M}(\alpha,X),Y).$
Hence, we can write $\ell(T_{F_M}(\alpha,X),Y)= \ell(T_{F_M}(\alpha_{A/B},\pi(X))$, in view of Lemma~\ref{pi}.
So, Lemma~\ref{lem:ell}\eqref{ell1} and the inclusion $Y \subseteq T_{F_M}(\alpha,Y)$ allow us to conclude that 
$$
\ell(T_{F_M}(\alpha,X),\ell(T_{F_M}(\alpha,Y)))\leq \ell(T_{F_M}(\alpha,X),Y) = \ell(T_{F_M}(\alpha_{A/B},\pi(X)).
$$
Combining this inequality with \eqref{-3-} and Lemma~\ref{D10}, we obtain 
\begin{equation}\label{-3-bis}
\frac{\ell(T_{F_n}(\alpha,X),\ell(T_{F_n}(\alpha,Y)))}{\card{F_n}} \leq H_{alg}(\alpha_{A/B}, \pi(X)) + \eps\leq \halg(\alpha_{A/B})+\eps
\end{equation}
for all $n \geq M$.  In view of Lemma~\ref{D10} again, there exists $M^*\geq M$ such that, for every $n\geq M^*$,
\begin{equation}\label{-4-}
\frac{\ell(T_{F_n}(\alpha, Y))}{\card{F_n}} \leq H_{alg}(\alpha, Y)  + \eps\leq \halg(\alpha_B)+\eps.
\end{equation}

By Lemma~\ref{D10}, 
\begin{equation}\label{-6-}
H_{alg}(\alpha,X)=\inf_{n\in\N}\frac{\ell(T_{F_n}(\alpha,X))}{\card{F_n}}\leq \frac{\ell(T_{F_{M^*}}(\alpha, X))}{\card{F_{M^*}}},
\end{equation}
and by Lemma~\ref{lem:ell}\eqref{ell2},
\begin{equation}\label{-7-}
\ell(T_{F_{M^*}}(\alpha, X))\leq \ell(T_{F_{M^*}}(\alpha,X) , T_{F_{M^*}}(\alpha, Y))+\ell(T_{F_{M^*}}(\alpha,Y)).
\end{equation}
Therefore, by \eqref{-6-}, \eqref{-7-}, \eqref{-3-bis} and \eqref{-4-},
\[
H_{alg}(\alpha,X)\leq \frac{\ell(T_{F_{M^*}}(\alpha, X))}{\card{F_{M^*}}}\leq \leq \frac{\ell(T_{F_{M^*}}(\alpha,X) , T_{F_{M^*}}(\alpha, Y))}{\card{F_{M^*}}} +\frac{\ell(T_{F_{M^*}}(\alpha,Y))}{\card{F_{M^*}}}\leq  \halg(\alpha_{A/B}) + \halg(\alpha_B) + 2 \eps.
\]
Since $\eps$ was chosen arbitrarily, this proves \eqref{leq}.
\end{proof}

\subsection{Application to the topological entropy via the Bridge Theorem}\label{BT-sec}

We recall that, inspired by the work of Kolmogorov and Sinai in ergodic theory, Adler, Konheim and McAndrew \cite{AKM} introduced the topological entropy for continuous selfmaps of compact topological spaces, while a different notion of topological entropy for uniformly continuous selfmaps of metric spaces was given by Bowen \cite{B} and Dinaburg \cite{Din} independently.
In the realm of topological groups, Yuzvinski \cite{Y} proved the so-called Addition Theorem (usually called Yuzvinski's addition formula) for the topological entropy of continuous endomorphisms of compact metrizable groups, that was recently extended to all compact groups in \cite{Dik+Manolo}.  

Lind, Schmidt and Ward \cite{LSW} generalized for $\Z^d$-actions on compact metrizable groups both the definition of topological entropy by Bowen, as well as that by Adler, Konheim and McAndrew, showing that they coincide. They proved the Addition Theorem for $\Z^d$-actions on compact metrizable groups. 
Moreover, Ollagnier \cite{Oll} defined the topological entropy for amenable group actions on compact spaces using open covers as in \cite{AKM}.
Recently, Li \cite{Li} established the Addition Theorem for actions of countable amenable groups $G$ on compact metrizable groups $K$; see also Chapter 13 in the recent monograph of Kerr and Li \cite{KL}.
Even if a proof seems to be not available in the literature, this result can be apparently extended to the general case, that is, without the assumption on $G$ to be countable and on $K$ to be metrizable.

Recently, Ceccherini-Silberstein, Coornaert and Krieger \cite{CCK} extended Ornstein-Weiss Lemma from \cite{OW} to cancellative amenable semigroups, and this allowed them
to define the topological entropy for amenable semigroup actions on compact spaces as follows. Let $C$ be a compact topological space, let $S$ be a cancellative right amenable semigroup, and consider a right action
$C\overset{\gamma}{\curvearrowleft}S$ by continuous selfmaps.
Let $\mathcal U=\{U_j\}_{j\in J}$ and $\mathcal V=\{V_k\}_{k\in K}$ be two open covers of $C$. One says that $\mathcal V$ refines $\mathcal U$, denoted by $\mathcal U\prec\mathcal V$, if for every $k\in K$ there exists $j\in J$ such that $V_k\subseteq U_j$. Moreover, let 
$$\mathcal U\vee\mathcal V=\{U_j\cap V_k\colon {(j,k)\in J\times K}\}.$$
Let also $N(\mathcal U)=\min\{n\in\N_+\colon\mathcal U\ \text{admits a subcover of size $n$}\}.$

For a continuous selfmap $f:C\to C$ and an open cover $\mathcal U$ of $C$, let $f^{-1}(\mathcal U)=\{f^{-1}(U_j)\}_{j\in J}$, and 
for  $F\in\Pf(S)$, let $$\mathcal U_{\gamma,F}=\bigvee_{s\in F}\gamma(s)^{-1}(\mathcal U).$$

The \emph{topological entropy of $\gamma$ with respect to $\mathcal U$} is given by the limit
$$
H_{top}(\gamma,\mathcal U)=\lim_{i\in I}\frac{\log N(\mathcal U_{\gamma,F_i})}{|F_i|},
$$ where $(F_i)_{i\in I}$ is a right F\o lner net of $S$; this limit exists and does not depend on the choice of  $(F_i)_{i\in I}$ by \cite[Theorem 1.1]{CCK}. The \emph{topological entropy of $\gamma$} is $$h_{top}(\gamma)=\sup\{H_{top}(\gamma,\mathcal U)\colon \mathcal U\ \text{open cover of}\ C\}.$$

Weiss \cite{W} and Peters \cite{P1} discovered a remarkable connection, usually named Bridge Theorem, between the topological entropy and the algebraic entropy, which was proved in general in \cite{DGB1}. More precisely, the topological entropy of a continuous endomorphism $\phi$ of a compact abelian group $K$ coincides with the algebraic entropy of its dual endomorphism $\widehat \phi$ of the Pontryagin dual $\widehat K$ of $K$, which is a discrete abelian group. This connection was extended to totally disconnected locally compact abelian groups in \cite{DGB2}.

The Bridge Theorem from \cite{P1} was recently extended by Kerr and Li \cite{KL} to actions of countable amenable groups on compact metrizable abelian groups. Then Virili \cite{V2} proved it for actions of amenable groups on locally compact abelian groups. Moreover, the one from \cite{DGB2} was extended in \cite{GB} to semigroup actions on totally disconnected locally compact abelian groups.
In \cite{DFG-amac}, generalizing the main result of \cite{W}, we proved a Bridge Theorem for left actions of cancellative left amenable monoids on totally disconnected compact abelian groups (their Pontryagin dual groups are precisely the torsion abelian groups).


To state those results in details, for a locally compact abelian group $A$, denote by $\widehat A$ its Pontryagin dual. For a continuous homomorphism $\phi:A\to B$, where $B$ is another locally compact abelian group, let $\widehat \phi:\widehat B\to \widehat A$ be the dual of $\phi$, defined by $\widehat\phi(\chi)=\chi\circ\phi$ for every $\chi\in\widehat B$.

Moreover, let $S$ be a cancellative  right amenable semigroup, $K$ a compact abelian group and $C\overset{\gamma}{\curvearrowleft}S$ be a right action by continuous endomorphisms. 
Then $\gamma$ induces the left action $S \overset{\widehat\gamma}{\curvearrowright}\widehat K$ by endomorphisms (called the \emph{dual action} of $\gamma$), defined by 
$$\widehat\gamma(s)=\widehat{\gamma(s)}:\widehat K\to \widehat K\quad \text{for every}\ s\in S.$$ 

Analogously, let $S$ be a cancellative right amenable semigroup, $A$ an abelian group, and $S\overset{\alpha}{\curvearrowright} A$ a left action by endomorphisms. Then $\alpha$ induces the right action $\widehat A\overset{\widehat\alpha}{\curvearrowleft}S$, defined by
$$
\widehat\alpha(s)=\widehat{\alpha(s)}:\widehat A\to \widehat A,\quad \text{for every}\ s\in S.
$$
Note that by Potryagin duality $\widehat{\widehat\gamma}=\gamma$ and $\widehat{\widehat\alpha}=\alpha$ up to conjugation (due to canonical isomorphisms). 

\smallskip
The following theorem combines the Bridge Theorems from \cite{DFG} and \cite{V2}.

\begin{theorem}\label{BTalg} 
Let $S$ be a cancellative right amenable semigroup, $A$ an abelian group and $S\overset{\alpha}{\curvearrowright}A$ a left action. Then 
$h_{alg}(\alpha)=h_{top}(\widehat\alpha)$
in the following two cases: 
\begin{enumerate}[(a)]
\item $A$ is torsion;
\item $S$ is a group.
\end{enumerate}
\end{theorem}

Since the Pontryagin dual of a compact abelian group $K$ is torsion precisely when $K$ is totally disconnected, we immediately obtain the following counterpart of the above theorem.

\begin{corollary}\label{BTcor}
Let $S$ be a cancellative right amenable semigroup, $K$ a compact abelian group and $K\overset{\gamma}{\curvearrowleft}S$ a right action. Then $h_{top}(\gamma)=h_{alg}(\widehat\gamma)$
in the following two cases: 
\begin{enumerate}[(a)]
\item $K$ is totally disconnected;
\item $S$ is a group. 
\end{enumerate}
\end{corollary}

From Corollary~\ref{BTcor} and Theorem~\ref{ATintro}, we obtain the following Addition Theorem for the topological entropy.


\begin{corollary} 
Let $S$ be a locally monotileable cancellative right amenable monoid, $K$ a compact abelian group and $K\overset{\gamma}{\curvearrowleft}S$ a right action.
Let $H$ be a $\gamma$-invariant closed subgroup of $K$, and denote by $\gamma_{K/H}$ and $\gamma_H$ the induced actions of $S$ on $K/H$ and on $H$, respectively. If either $S$ is a group or $K$ is totally disconnected, then
\[h_{top}(\gamma)=h_{top}(\gamma_{K/H})+ h_{top}(\gamma_H).\]
\end{corollary}
\begin{proof}
Let $A=\widehat K$,  $\alpha=\widehat \gamma$ and let $B$ be the annihilator of $H$ in $A$.
By Corollary~\ref{BTcor} and Theorem~\ref{ATintro}, we have 
\begin{equation}\label{topalg}
h_{top}(\gamma)=h_{alg}(\alpha)=h_{alg}(\alpha_B)+h_{alg}(\alpha_{A/B}).
\end{equation}
By Pontryagin duality $B\cong\widehat{K/H}$ and $A/B\cong\widehat H$, and moreover these natural isomorphisms witness that $\alpha_B$ is conjugated to $\widehat{\gamma_{K/H}}$ and $\alpha_{A/B}$ is conjugated to $\widehat{\gamma_{H}}$. The algebraic entropy is invariant under conjugation (see \cite{DFG}), so, by applying also Corollary~\ref{BTcor}, we get $h_{alg}(\alpha_B)=h_{top}(\gamma_{K/H})$ and $h_{alg}(\alpha_{A/B})=h_{top}(\gamma_H)$. We conclude by applying the last two equalities in \eqref{topalg}.
\end{proof}

\begin{remark}
(a) Our notion of local monotileability is inspired by that of monotileability from \cite{Weiss}. In both cases these are ``left'' conditions. Indeed, one could define that, for subsets $T,V$ of a semigroup $S$,  $T$ is a \emph{right monotile} of $V$ if there exists a subset $C$ of $S$ such that $\{Tc\colon c\in C\}$ is a partition of $V$.

Then, in a monoid $S$; a sequence $(F_n)_{n \in \N}$ in $\Pf(S)$ is \emph{right locally monotileable} if $F_0=\{1\}$ and $F_n$ is a right monotile of $F_{n+1}$ for every $n \in \N$.
So, a left amenable monoid is \emph{right locally monotileable} if it admits a right locally monotileable left F\o lner sequence.

\smallskip
(b) Assume that one would like to consider  the topological entropy of left actions of cancellative left amenable semigroups on compact spaces as in \cite{CCK}, or the algebraic entropy of right actions of cancellative left amenable semigroups on abelian groups. Then one should consider left amenable semigroups that are right locally monotileable, to obtain the counterparts of the above results.
\end{remark}

\section{Locally monotileable monoids}\label{Tileable-sec}

\subsection{Starting examples}

We propose some basic examples.

\begin{example}\label{finitemonot}
Every finite monoid $S$ is congruent monotileable, and so locally monotileable. This is witnessed by the sequence $(F_n)_{n\in\N}$ with $F_0=\{1\}$ and $F_n=S$ for all $n\in\N_+$, which is obviously a congruent and exhaustive right F\o lner sequence of $S$.
\end{example}

\begin{remark}\label{tire1}
Let $(F_n)_{n\in\N}$ be a \tileable\  sequence of a cancellative monoid $M$ and consider a tiling sequence $(K_n)_{n\in\N}$ associated to $(F_n)_{n\in\N}$. Then it is easy to prove by induction that $|F_n|=\prod_{i=1}^{n}|K_i|$ for every $n\in\N$ and that $|K_n \ldots K_l| = \prod_{i=l}^{n} |K_i|$ for all positive integers $l\leq n$. 
\end{remark}

A strictly increasing sequence of natural numbers $(a_n)_{n\in\N}$ is an \emph{a-sequence} if $a_0=1$ and $a_n\mid a_{n+1}$ for every $n\in\N$.
It follows from Remark~\ref{tire1} that if $(F_n)_{n\in\N}$ is \tileable\ sequence of a monoid $S$, the sequence $(|F_n|)_{n\in\N}$ is an a-sequence.

\begin{example}\label{esemono}
\begin{enumerate}[(a)]
\item The monoid $(\mathbb{N},+)$ is congruent monotileable (so locally monotileable) and $\MT$. 
Indeed, consider an a-sequence $(a_n)_{n\in\N}$ (for example the sequence $(n!)$) and define $F_n=[0,a_n-1]$ for every $n\in\N$. Then
each finite subset of $\N$ is contained in some $F_n$, and the sequence $(F_n)_{n\in\N}$ is a congruent and exhaustive F\o lner sequence of $\N$ consisting of monotiles.  

\item  Clearly, by (a) the sequence $([0,n!-1])_{n\in\N}$ is a congruent and exhaustive F\o lner sequence of $\N$, while $([0,n!-1])_{n\in\N}$  is a congruent F\o lner sequence of $\Z$ that is not exhaustive. 


\item The sequence $([0,n])_{n\in\N}$ is a F\o lner sequence of $\N$ that is not \tileable, while $[0,n]$ is a monotile of $\N$ for every $n\in\N$.
\item  On the other hand, consider the sequence $(F_n)_{n\in\N}$, where $F_0=\{0\}$ and
$F_n=\{0\}\cup [2,3(2^{n-1}-1)] \cup \{3(2^{n-1} + 1)\}$ for every $n\in\N_+$.
For every $n\in\N_+$, we have $F_{n+1} = F_n \sqcup 3(2^{n-1} - 1)+3 + F_n)$. Then $(F_n)_{n\in\N}$ is a locally monotileable F\o lner sequence of $\N$ but $F_n$ is not a monotile of $\N$ for every $n\in\N_+$.
\end{enumerate}	
\end{example}

%


The following example shows a \tileable\ F\o lner sequence $(F_n)_{n\in\N}$ of $\Z$ such that, for every tiling sequence $(K_n)_{n\in\N}$ associated to $(F_n)_{n\in\N}$, $0\not\in K_n$ for every $n\geq 2$; so $(F_n)_{n\in\N}$ cannot be congruent, even if $\Z$ is congruent monotileable. Note that in this case $F_n\subseteq F_{n+1}$ for every $n\in\N$.

\begin{example}\label{0notin}
In $\Z$ let $F_0=\{0\}$ and, for every $n\in\N_+$, let $F_n=[-2^{n-1}+1,2^{n-1}].$
Then, $K_1=F_1=\{0,1\}$, but for every $n\in\N$ with $n\geq 2$,
$K_n=\{-2^{n-2},2^{n-2}\}.$
In particular, $0\not\in K_n$ for every $n\geq 2$.
\end{example}

We end this section showing that all countable locally finite groups are locally monotileable.

\begin{proposition}\label{monotor}
If $G$ is a countable locally finite group, then $G$ is congruent monotileable, and so \tileable.
\end{proposition}
\begin{proof}
The group $G$ is countable so we enumerate its elements as $G=\{g_n\colon n\in\N\}$ with $g_0=1$. For every $n\in\N$, let $F_n=\langle g_1,\dots,g_n\rangle$. Since $G$ is locally finite, all subgroups $F_n$ are finite. Clearly, $F_n\subseteq F_{n+1}$ for every $n\in\N$.
	
Fix $g_{\bar n}\in G$. If $n\geq\bar n$, then $g_{\bar n}\in F_n$ and so $F_ng_{\bar n}=F_n$, hence $\left| F_ng_{\bar n}\setminus F_n\right|/\left| F_n\right|=0$. Therefore,
\begin{equation}\label{eqtor1}
\lim_{n\to\infty}\frac{\left|F_ng_{\bar n}\setminus F_n\right|}{\left|F_n\right|}=\lim_{k\to\infty}\frac{\left| F_{\bar n+k}g_{\bar n}\setminus F_{\bar n+ k}\right|}{\left| F_{\bar n+k}\right|}=0.
\end{equation}
The left cosets of $F_n$ in $F_{n+1}$ are disjoint left translates of $F_n$ and they cover $F_{n+1}$. By this and by \eqref{eqtor1} we conclude that $(F_n)_{n\in\N}$ is a congruent (hence, \tileable) exhaustive right F\o lner sequence of $G$.	
\end{proof}

\subsection{Basic properties}

Here we start with the following property of F\o lner sequences with respect to translates of finite sets. It applies in both Proposition~\ref{lm=cm} and Proposition~\ref{MTA->MT}.

\begin{lemma}\label{letr1}
Let $S$ be a cancellative monoid and let $(F_n)_{n\in\N}$ be a right F\o lner sequence of $S$. If $X\in\Pf(S)$, then there exists $\bar n\in\N$ such that, for every $n>\bar n$, $F_n$ contains a left translate of $X$.    
\end{lemma}	
\begin{proof}
Let $X\in\Pf(S)$. Since $(F_n)_{n\in\N}$ is a right F\o lner sequence of $S$,
\begin{equation}\label{eqtr1}
\lim_{n\to\infty}\frac{\lvert F_nX\setminus F_n\rvert}{\lvert F_n\rvert}=0.
\end{equation}
We assume by contradiction that there is an increasing sequence of natural numbers $(k_n)_{n\in\N}$ such that, for all $n\in\N$, each $F_{k_n}$ contains no left translates of $X$. Fix a $n\in\N$. Then, for all $f\in F_{k_n}$, the set $fX\setminus F_{k_n}$ is non-empty. Therefore, we can define a map $\phi_n:F_{k_n}\rightarrow X$, $f\mapsto \phi_n(f)$ with $f\phi_n(f)\in fX\setminus F_{k_n}$.
By the pigeonhole principle there exists $x\in X$ such that $\lvert \phi_n^{-1}(x)\rvert\geq \lvert F_{k_n}\rvert/\lvert X\rvert$.
Clearly, if $f_1$ and $f_2$ are two distinct elements of $\phi_n^{-1}(x)$, then $f_1x\neq f_2x$. This implies that
\begin{equation*}
\lvert F_{k_n}X\setminus F_{k_n}\rvert\geq \lvert \phi_n^{-1}(x)\rvert\geq \frac{\lvert F_{k_n}\rvert}{\lvert X\rvert}.
\end{equation*}
Dividing both sides by $\lvert F_{K_n}\rvert$, we obtain $\lvert F_{k_n}X\setminus F_{k_n}\rvert/\lvert F_{k_n}\geq 1/\lvert X\rvert$.
Since this holds for all $n\in\N$,
\begin{equation*}
\lim_{n\to\infty}\frac{\lvert F_{k_n}X\setminus F_{k_n}\rvert}{\lvert F_{k_n}\rvert}\geq \frac{1}{\lvert X\rvert},
\end{equation*}
that is in contradiction with \eqref{eqtr1}.
\end{proof}

We proceed with the following observation on monotiles.

\begin{remark}\label{retr1}
Let $G$ be a monoid and let $T$ be a monotile of $G$. 
\begin{enumerate}[(a)]
\item If $G$ is a group, then $gT$ is still a monotile of $G$ for all $g\in G$.
\item Item (a) may fail in case $G$ is not a group. For example, if $G=\N$, then a monotile $T$ of $G$ necessarily contains $0$, so $g+T$ is not a monotile of $G$ if $g \ne 0$.
\end{enumerate}
\end{remark}

 We omit the easy proof of the next two lemmas (the first one can be used for a proof of the second one, as well as further on). 

\begin{lemma}\label{disun}
Let $G$ be a group. Consider $X\subseteq Y\subseteq Z$ subsets of $G$. If there exist $U$ and $V$ such that
$Y=\bigsqcup_{u\in U}uX$ and $Z=\bigsqcup_{v\in V}vY,$ then $Z=\bigsqcup_{v\in V,\ u\in U}vuX=\bigsqcup_{t\in T}tX$, where $T=VU$.
\end{lemma}

%

This lemma simply asserts that if $X$ is a monotile of $Y$ and $Y$ is a monotile of $Z$ then $X$ is a monotile of $Z$.

\begin{lemma}\label{tilesub}
Let $M$ be a cancellative monoid and $(F_n)_{n\in\N}$ a \tileable\ sequence of $M$. Then every subsequence $(F_{k_n})_{n\in\N}$ is still a \tileable\ sequence  of $M$.
\end{lemma}

%
%
%

Next we see that the class of \tileable\ cancellative monoids if stable under taking countable direct sums.

\begin{proposition}\label{monosum}
Let $(G_i)_{i\in\mathbb{N}_+}$ be a family of countable \tileable\  cancellative monoids, then $G=\bigoplus_{i\in\mathbb{N}_+} G_i$ is \tileable.
\end{proposition}
\begin{proof}
Since $G_i$ is \tileable\ for all $i\in\mathbb{N}_+$, we can fix a \tileable\ right F\o lner sequence $(F_{i,n})_{n\in \N}$ of $G_i$. We define a new sequence by letting, for every $n\in\N$, $F_n=\bigoplus_{i=1}^n F_{i,n}\oplus\bigoplus_{i\geq {n+1}}\{1\}$.

First we prove that $(F_n)_{n\in \N}$ is a right F\o lner sequence of $G$. Fix $\varepsilon>0$ and $h=(h_n)_{n\in\mathbb{N}}\in\bigoplus_{n\in\mathbb{N}} G_n$, and let $k=\max\{n\in\N\colon h_n\neq1\}\in\mathbb{N}$. For every $i\in\{1,\ldots,k\}$ there exists $m_i\in\N$ such that for all $n\geq m_i$,
\begin{equation}\label{eqms1}
\frac{\left|F_{i,n}h_i\setminus F_{i,n}\right|}{\left| F_{i,n}\right|}<\frac{\varepsilon}{k}.
\end{equation}
If $n\geq k$, then $F_nh=\bigoplus_{i=1}^n F_{i,n}h_i\oplus\bigoplus_{i\geq {n+1}}\{1\}$. Moreover, since 
$$F_nh\setminus F_n \subseteq \bigcup_{i=1}^{n} \left((F_{i,n}h_i\setminus F_{i,n} )\oplus \bigoplus_{j=1,j\neq i}^n F_{j,n}h_j \right)\oplus\bigoplus_{i\geq {n+1}}\{1\},$$ 
and since $\left|F_{i,n}h_i\setminus F_{i,n}\right|=0$ for every $k<i\leq n$, we obtain that
\begin{equation}\label{eqms4}
\left|F_nh\setminus F_n\right|\leq\sum_{i=1}^{n}\left(\left|F_{i,n}h_i\setminus F_{i,n}\right|\prod_{j=1,j\neq i}^n\left| F_{j,n}h_j\right|\right)=\sum_{i=1}^{k}\left(\left|F_{i,n}h_i\setminus F_{i,n}\right|\prod_{j=1,j\neq i}^n\left| F_{j,n}h_j\right|\right).
\end{equation}	
By \eqref{eqms1} and \eqref{eqms4}, observing that $\left| F_{j,n}h_j\right|=\left| F_{j,n}\right|$ for every $j\in\N_+$ (as the monoid $G_j$ is cancellative)
and that $\left|F_n\right|=\prod_{i=1}^{n}\left|F_{i,n}\right|$ for every $n\in\N$, we conclude that for every $n\geq m:=\max\{m_i\colon 1\leq i\leq k\}$,
\begin{equation*}
\begin{split}
\frac{\left|F_{n}h\setminus F_{n}\right|}{\left| F_{n}\right|}&\leq \frac{\sum_{i=1}^{k}\left(\left|F_{i,n}h_i\setminus F_{i,n}\right|\prod_{j=1,j\neq i}^n\left| F_{j,n}h_j\right|\right)}{\prod_{i=1}^{n}\left| F_{i,n}\right|}\leq \\
&\leq\frac{\sum_{i=1}^{k}\left(\left|F_{i,n}h_i\setminus F_{i,n}\right|\prod_{j=1,j\neq i}^n\left| F_{j,n}\right|\right)}{\prod_{i=1}^{n}\left| F_{i,n}\right|}\leq\sum_{i=1}^k\frac{\left| (F_{i,n}h_i)\setminus F_{i,n}\right|}{\left| F_{i,n}\right|}<\varepsilon.
\end{split}
\end{equation*}
Then, when $\varepsilon\to 0$, we have $\lim_{n\to\infty}{\left|F_{n}h\setminus F_{n}\right|}/{\left| F_{n}\right|}=0$, and hence, $(F_n)_{n\in\N}$ is a right F\o lner sequence of $G$.

By hypothesis $(F_{i,n})_{n\in\N}$ is \tileable\ for every $i\in\N_+$. Then, for all $n\in\N$ and $i\in\N_+$ there exists a finite set $\widetilde G_{i,n}\subseteq G_i$ such that $F_{i,n}=\bigsqcup_{g_i\in \tilde{G}_{i,n}} (g_iF_{i,n-1})$.
Define $\tilde G_n=\{g=(g_i)_{i\in\N}\in G\colon g_i\in \widetilde G_{i,n}\text{ if }i\leq n,\ g_i=1\text{ if }i>n\}$.
Then $\widetilde G_{n}$ is finite and $F_{n}=\bigsqcup_{g\in\widetilde G_{n}}(gF_{n-1})$. Therefore, $(F_n)_{n\in\N}$ is a \tileable\ right F\o lner sequence of $G$.
\end{proof}

By Proposition~\ref{monosum}, Example~\ref{finitemonot} and Example~\ref{esemono}, we have the following consequence.

\begin{corollary}\label{monofingen}
$\M$ is stable under countable direct sums. In particular, $\M$ contains all finitely generated abelian groups.
\end{corollary}

The first assertion of this corollary can be obtained also as a consequence of a more general result proved below (see Theorem~\ref{tileable}).

\subsection{Relations among notions of monotileability}

We start this part verifying that an exhaustive congruent sequence of a monoid $S$ consists necessarily of monotiles of $S$.

\begin{lemma}\label{claim*}
If $S$ is a cancellative monoid and $(F_n)_{n\in\N}$ an exhaustive congruent sequence of $S$, then each $F_n$ is a monotile of $S$.
In particular, every countable congruent monotileable monoid is $\MTA$.
\end{lemma}
\begin{proof} 
Fix $n_0\in\N$. We prove that $F_{n_0}$ is a monotile of $S$.

For every $m\geq n_0$ let $P_m= K_m\ldots K_{n_0}$ and $K=\bigcup_{n> n_0}P_n$. 
In order to prove our assertion it suffices to prove the equality 
\begin{equation}\label{EqDic1}
S=\bigsqcup_{k\in K}kF_{n_0}.
\end{equation}
To check that the above union is pairwise disjoint pick $k,k'\in K$ with $k\neq k'$; then there exists $m>n_0$ such that $k,k'\in P_{m}$, and so $kF_{n_0}\cap k'F_{n_0}=\emptyset$.
To prove the equality in \eqref{EqDic1} pick $x\in S$. Since the sequence $(F_n)_{n\in\N}$ is exhaustive $M=\bigcup_{n\in\N}F_n$, and so there exists $k\in\N$ such that $x\in F_k$. If $k\leq n_0$, then $x\in F_{n_0}$, as $(F_n)_{n\in\N}$ is a congruent (so increasing) sequence, and we are done. If $k>n_0$, then $F_k=P_kF_{n_0}\subseteq KF_{n_0}$. Hence, we have the required equality in \eqref{EqDic1}. 
\end{proof}

Next we give an example of a commutative (so amenable) monoid that is locally monotileable but neither $\MTA$ nor $\MT$, and so not congruent monotileable by Lemma~\ref{claim*}.

\begin{example}\label{New:Example} 
Consider the submonoid $S=(\N\times\N)\setminus(\{0\}\times\N_+)$ of $\N\times\N$. 
\begin{enumerate}[(a)]
\item  We prove here that if $F$ is a finite monotile of $S$ and $(u,v_1),(u,v_2)\in F$, then $v_1=v_2$.
As $F$ is a monotile of $S$, there exists $C\subseteq S$ such that 
\begin{equation}\label{Neeew}
S=\bigsqcup_{c\in C}c+F.
\end{equation} 
Since $F$ is finite there exists $k\in\N_+$ such that $(1,k+v_1),(1,k+v_2)\in S\setminus F$. 
By \eqref{Neeew}, there exist $s_1,s_2\in C$ with $(1,k+v_1)\in s_1+F$ and $(1,k+v_2)\in s_2+F$.
As $(1,k+v_1),(1,k+v_2)\in S\setminus F$, one has $s_1\ne (0,0) \ne s_2$. Then $s_1=(1,k+v_1)$ and $s_2=(1,k+v_2)$, as $(\{0\}\times \N)\cap  S =\{(0,0)\}$. Then $(1+u,k+v_1+v_2)\in (s_1+F)\cap (s_2+F)$, which implies that $s_1=s_2$. Therefore, $v_1=v_2$, as required.

\item The monoid $S$ is not $\MTA$. Suppose by contradiction that there exists a F\o lner sequence $(F_n)_{n\in\N}$ of $S$ such that $F_n$ is a monotile of $S$ for every $n\in\N$. Since $(F_n)_{n\in\N}$ is a F\o lner sequence of $S$, without loss of generality we have, for all sufficiently large $n\in\N$,
\begin{equation*}
\left|F_n+(1,1)\cap F_n\right|>\frac{n-1}{n}|F_n|>\frac{1}{2}|F_n|\quad\text{and}\quad\left|F_n+(1,0)\cap F_n\right|>\frac{n-1}{n}|F_n|>\frac{1}{2}|F_n|.
\end{equation*}
This implies that $\left|(F_n+(1,1))\cap (F_n+(1,0))\right|>0$ and so that there exists $(u,v)\in F_n$ such that also $(u,v+1)\in F_n$. This contradicts (a), since $F_n$ is a monotile of $S$ by assumption.

\item The monoid $S$ is not $\MT$. Indeed, according to item (a), the finite set $\{(0,1),(0,1)\}$ is not contained in any monotile of $S$. 

\item The monoid $S$ is locally monotileable. Let $(F_n)_{n\in\N}$ be the sequence in $\Pf(S)$ given by $F_0=\{(0,0)\}$ and $F_n=[2^{n-1},2^{n})\times[2^{n-1},2^{n})$ for all $n\in\N_+$.
It is easy to see that $(F_n)_{n\in\N}$ is a locally monotileable F\o lner sequence of $S$.
\end{enumerate}
\end{example}

Next we see that for groups the new notion of local monotileability coincides with the existing one of congruent monotileability.

\begin{proposition}\label{lm=cm}
A countable amenable group $G$ is locally monotileable if and only if $G$ is congruent monotileable.
\end{proposition}
\begin{proof} 
By definition if $G$ is congruent monotileable, then $G$ is locally monotileable.

Assume that $G$ is locally monotileable and let $(F_n)_{n\in\N}$ be a locally monotileable right F\o lner sequence of $G$.
Since $G$ is countable, enumerate its elements as $G=\{g_n\colon n\in\N\}$ with $g_0=1$. For every $n\in\N$, let $G_n=\{g_0,\dots,g_n\}$.
We build inductively a sequence $(H_n)_{n\in\N}$ in $\Pf(G)$ satisfying the following conditions:
\begin{enumerate}[(1)]
\item $H_0=\{1\}$;
\item for all $n\in\N$, $g_n\in H_n$;
\item for all $n\in\N$, there exist $g\in G$ and $m\in\N$ such that $H_n=g F_m$;
\item for all $n\in\N$, there exists $K_{n+1}$ with $1\in K_{n+1}$ and $H_{n+1}=\bigsqcup_{k\in K_{n+1}}kH_n$.
\end{enumerate}
Then (2) and (3) ensure that $(H_n)_{n\in\N}$ is a right F\o lner sequence of $G$, (2) that it is exhaustive, while (4) shows that $(H_n)_{n\in\N}$ is congruent and implies that $H_n\subseteq H_{n+1}$ for all $n\in\N$. 

\smallskip
Suppose that we already defined $H_0,\dots,H_n$. Then $H_n = tF_m$ for some $t\in G$ and $m\in\N$. There exists $l\in \N$ such that $l>m$ and 
\begin{equation}\label{eqtr2}
\lvert F_lg_{n+1}\setminus F_l\rvert\leq\frac{\lvert F_l\rvert}{\lvert H_n\rvert}.
\end{equation}
Note that there exists $D\in\Pf(G)$ such that $F_l=\bigsqcup_{d\in D}dF_m$. Therefore, 
\(F_l=\bigsqcup_{d'\in D'}d'H_n,\)
where $D'=Dt^{-1}$. 

We show that there exists $\bar d\in D'$ such that $\bar d H_ng_{n+1} \subseteq F_l$.
Otherwise, for every $d'\in D'$, there would exist a point $v\in d'H_ng_{n+1}\setminus F_l$; since the sets $d'H_ng_{n+1}$, for $d'\in D'$, are pairwise disjoint, and are contained in $F_lg_{n+1}$, we would have that
\[\lvert F_lg_{n+1}\setminus F_l\rvert\geq \lvert D'\rvert=\frac{\lvert F_l\rvert}{\lvert H_n\rvert},\]
contradicting \eqref{eqtr2}.

Clearly $\bar d H_ng_{n+1} \subseteq F_l$, in conjunction with $1\in H_n$, yields $\bar d g_{n+1}\in F_l$. Define $H_{n+1}={\bar d}^{-1}F_l$ and $K_{n+1}= \bar d^{-1}D'$.
Then $1=\bar d^{-1}\bar d\in K_{n+1}$ and 
\begin{equation*}
H_{n+1}={\bar d}^{-1}F_l={\bar d}^{-1}\bigsqcup_{d'\in D'}d'H_n=\bigsqcup_{k\in K_{n+1}}kH_n.\qedhere
\end{equation*}
\end{proof}

The next result shows that for groups $\MTA$ implies $\MT$.

\begin{proposition}\label{MTA->MT}
If a countable group $G$ is $\MTA$, then $G$ is $\MT$.
\end{proposition}
\begin{proof}
By hypothesis, there exists a right F\o lner sequence $(F_n)_{n\in\N}$ of $G$ such that $F_n$ is a monotile of $G$ for every $n\in\N$. Let $K\in \Pf(G)$. By Lemma~\ref{letr1} there exist $n\in\N$ and $g\in G$ such that $gK\subseteq F_n$. This immediately implies that $K\subseteq g^{-1}F_n$. By Remark~\ref{retr1} $g^{-1}F_n$ is a monotile of $G$ and this concludes the proof.
\end{proof}

The implication in the above proposition cannot be inverted. Indeed, every residually finite group is $\MT$, 
so free non-abelian groups are $\MT$ but fail to be amenable.

\section{Extension Theorem}

\subsection{CIF sequences}

Let $G$ be a countable amenable group and let $(G_n)_{n\in\N}$ be an increasing exhaustive sequence in $\Pf(G)$. 
A right F\o lner sequence $(F_n)_{n\in\N}$ of $G$ is \emph{canonically indexed} (briefly CIF sequence of $G$) with respect to $\mathcal G$ if for all $n\in\N_+$ and for all $g\in G_n$,
$$\frac{\lvert F_ng\setminus F_n\rvert}{\lvert F_n\rvert}<\frac{1}{n}.$$

\begin{proposition}\label{CIF}
Let $G$ be a countable amenable group and let $(G_n)_{n\in\N}$ be an increasing exhaustive sequence in $\Pf(G)$. For every right F\o lner sequence $(F_n)_{n\in\N}$ of $G$ there exists an increasing sequence of natural numbers $(k_n)_{n\in\N}$ (with $k_0=0$) such that $(F_{k_n})_{n\in\N}$ is a CIF sequence of $G$ with respect to $(G_n)_{n\in\N}$.
\end{proposition}
\begin{proof} 
By recursion we build a suitable sequence $({k_n})_{n\in\N}$ of natural numbers. Set $k_0=0$. Assume that $n>0$ and suppose that $k_{n-1}$ is defined. Then there exists $k_n\in\mathbb{N}$ such that $k_n> k_{n-1}$ and, for all $s\geq k_n$ and for all $g\in G_n$,  \(\left|F_sg\setminus F_s\right|/\left| F_s\right|< 1/n\).
Clearly, $(F_{k_n})_{n\in\N}$ is a CIF sequence with respect to $(G_n)_{n\in\N}$.
\end{proof}

\begin{remark}\label{remsub}
Let $G$ be a countable amenable group and let $(G_n)_{n\in\N}$ be an increasing exhaustive sequence in $\Pf(G)$.
\begin{enumerate}[(a)]
\item By definition, a CIF sequence $(F_n)_{n\in\N}$ of $G$ with respect to $(G_n)_{n\in\N}$ is also a right F\o lner sequence of $G$.
\item Let $G$ be a countable amenable group and $(G_n)_{n\in\N}$ be an increasing exhaustive sequence in $\Pf(G)$.  Then any subsequence of a CIF sequence with respect to $(G_n)_{n\in\N}$ is still a CIF sequence with respect to $(G_n)_{n\in\N}$. 
\item One can easily deduce from (a) and (b) that if $(F_n)_{n\in\N}$  is a CIF sequence of $G$ with respect to $(G_n)_{n\in\N}$, then for every increasing exhaustive sequence $(H_n)_{n\in\N}$ of $G$,  an appropriate subsequence $(F_{k_n})_{n\in\N}$  of $(F_n)_{n\in\N}$ is a CIF sequence of $G$ with respect to $(H_n)_{n\in\N}$, hence simultaneously a CIF sequence of $G$ with respect to 	$(G_n)_{n\in\N}$ as well. So, in this sense, the notion of ``CIF sequence with respect to $(G_n)_{n\in\N}$" essentially does not depend on $(G_n)_{n\in\N}$ and one can loosely speak of CIF sequence of $G$. 
\end{enumerate}
\end{remark}


The next lemma is straightforward to prove.

\begin{lemma}\label{tidisun}
Let $G$ and $K$ be countable amenable groups. Let $\pi\colon G\rightarrow K$ be a surjective homomorphism and $H=\ker \pi$. Let $\sigma\colon K\rightarrow G$ a section for $\pi$ with $\sigma(1_K)=1_G$. Consider a right F\o lner sequence $(E_n)_{n\in\N}$ of $H$ and a right F\o lner sequence $(F_n)_{n\in\N}$ of $K$ with $1_K\in F_n$ for all $n\in\N$. Then, for all $n\in\N$:
\begin{enumerate}[(a)]
\item $H \cap (\sigma(F_n)\sigma(F_n)^{-1}) = \{1\}$;
\item for all $i, j\in\N$, $E_j\sigma(F_i)=\bigsqcup_{f\in \sigma(F_i)}E_jf;$
\item for all $i, j\in\N$, $\lvert E_j\sigma(F_i)\rvert=\lvert E_{j}\rvert\lvert\sigma(F_i)\rvert$.
\end{enumerate} 
\end{lemma}
	%
%

The next theorem is implicitly contained in \cite[Theorem 2.27]{DFG-amac}, 
for the sake of completeness we provide a complete and independent proof. 

\begin{theorem}\label{teofolner}
Let $G$ and $K$ be countable amenable groups, $\pi\colon G\rightarrow K$ be a surjective homomorphism with $H=\ker \pi$ and $\sigma\colon K\rightarrow G$ be a section for $\pi$ with $\sigma(1_K)=1_G$. 
If  $(E_n)_{n\in\N}$ is a right F\o lner sequence of $H$ and $(F_n)_{n\in\N}$ is a right F\o lner sequence of $K$ with $1_K\in F_n$ for all $n\in\N$, then there exist  two increasing sequences of natural numbers $(m_n)_{n\in\N}$ and $(k_n)_{n\in\N}$ such that:
\begin{enumerate}[(1)]
\item the sequence given by
\begin{equation}\label{cifEq1}
\bar F_n=E_{m_n}\sigma(F_{k_n}) \text{ for all }n\in\N,
\end{equation}
is a right F\o lner sequence of $G$;
\item for all $n\in\N$,	$\lvert \bar F_n\rvert=\lvert E_{m_n}\rvert\lvert F_{k_n}\rvert$;
\item for every sequence $(a_n)_{n\in\N}$ in $\N$ there exists a sequence $(h_n)_{n\in\N}$ in $\N$ such that, for all $n\in\N$, $m_{h_{n+1}}>m_{h_n}+a_n$ and the sequence given by
\begin{equation}\label{cifEq4}
F^*_n=E_{m_{h_n}}\sigma(F_{k_n}) \text{ for all }n\in\N,
\end{equation}
is a right F\o lner sequence of $G$.
\end{enumerate}
\end{theorem}
\begin{proof}
The group $G$ is countable so we enumerate its elements as $G=\{g_n\colon n\in\N\}$ with $g_0=1$; let also $G_n=\{g_0,\dots,g_n\}$ for all $n\in\mathbb{N}$. By Proposition~\ref{CIF} there exists an increasing sequence of natural numbers $(k_n)_{n\in\N}$ (with $k_0=0$) such that $(F_{k_n})_{n\in\N}$ is a CIF sequence of $K$ with respect to $(\pi(G_n))_{n\in\N}$.

For every $n\in\N$ define
\begin{equation*} 
H_n=\{h\in H: h\sigma(F_{k_n})\cap \sigma(F_{k_n})G_n\neq\emptyset \}=\sigma(F_{k_n})G_n\sigma(F_{k_n})^{-1}\cap H.
\end{equation*}
Since $\sigma(F_{k_n})$ and $G_n$ are finite, also $\sigma(F_{k_n})G_n\sigma(F_{k_n})^{-1}$ is finite. This means that also $H_n$ is finite. 

Clearly, $\bigcup_{n\in\N}H_n\subseteq H$. 
On the other hand, for every $h\in H$, there is an $n_h\in\N$ such that $h\in G_{n_h}$. So $h\in H_{n_h}$, as $\sigma(1_k)=1_G$ and $1_K\in F_n$ for all $n\in\N$. Then $H=\bigcup_{n\in\N}H_n$ and $H_n\subseteq H_{n+1}$ for all $n\in\N$.

By Proposition~\ref{CIF} there exists an increasing sequence of natural numbers $(m_n)_{n\in\N}$ (with $m_0=0$) such that $(E_{m_n})_{n\in\N}$ is a CIF sequence of $H$ with respect to $(H_n)_{n\in\N}$. 
Let $(\bar F_n)_{n\in\N}$ be the sequence given by \eqref{cifEq1}. Fix arbitrarily $n\in \N$. Lemma~\ref{tidisun}(c) implies that $\lvert \bar F_n\rvert=\lvert E_{m_n}\rvert\lvert \sigma(F_{k_n})\rvert$. Since $\sigma$ is injective, we immediately obtain item (2). 
	
The rest of the proof is dedicated to verify that $(F_n)_{n\in\N}$ is a right F\o lner sequence. Fix $n\in\N$, pick an element $g\in G_n$ and let $C_n= \bar F_ng \setminus \bar F_n$, $A_n = \pi^{-1}(F_{k_n})$, so that 
$$C_n = (C_n \setminus A_n) \sqcup (C_n\cap A_n).	$$ 
Since we need to estimate $|C_n|$, it is enough to separately estimate  the cardinalities $|C_n \setminus A_n |$ and  $|C_n\cap A_n |$.
	
Pick an $x\in C_n$, then $x=e\sigma(f)g$ with $f\in F_{k_n}$ and $e\in E_{m_n}$. 
If $x\in C_n \setminus A_n$, then $\pi(x)=f\pi(g)\in F_{k_n}\pi(g)\setminus F_{k_n}$.
Since $(F_{k_n})_{n\in\N}$ is a CIF sequence with respect to $(\pi(G_n))_{n\in\N}$, we have $|F_{k_n}\pi(g)\setminus F_{k_n}|\leq {|F_{k_n}|}/{n}$, and so there are at most ${|F_{k_n}|}/{n}$ choices for $f$. Since $e$ ranges in $E_{m_n}$ arbitrarily, this leads to
\begin{equation}\label{eqsequ2}
|C_n \setminus A_n | \leq \frac{|E_{m_n}||F_{k_n}|}{n} \leq \frac{|\bar F_n|}{n},
\end{equation}
Now suppose that $x\in A_n$, i.e., $f':= \pi(x)=f \pi(g)\in F_{k_n}$. As 
$$\pi(\sigma(f'))= f' = f\pi(g) = \pi(\sigma(f)g),$$
we get $\pi(\sigma(f)g\sigma(f')^{-1})=1$. Therefore,  $h_f:= \sigma(f)g\sigma(f')^{-1}\in H$. Actually, $h_f \in H_n$, as $g \in G_n$. From $\sigma(f)g =h_f\sigma(f')$ we deduce that $x=e\sigma(f)g= eh_f\sigma(f')$. As $x=(eh_f)\sigma( f')\notin\bar F_n = E_{m_n}\sigma(F_{k_n}) $, while $\sigma(f') \in \sigma(F_{k_n})$, we conclude that $eh_f\notin E_{m_n}$. Hence, $eh_f\in E_{m_n}h_f \setminus E_{m_n}$, so $x\in (eh_f)\sigma(f') \in (E_{m_n}h_f \setminus E_{m_n})\sigma(f')$. In view of $f' = f \pi(g)$ this proves that
$$C_n \cap A_n\subseteq \bigcup_{f\in F_{k_n}} ( E_{m_n}h_f \setminus E_{m_n})\sigma(f \pi(g)).$$
Since $(E_{m_n})_{n\in\N}$ is a CIF sequence with respect to $(H_n)_{n\in\N}$ and $h_f\in H_n$, we get $ | E_{m_n}h_f \setminus E_{m_n}| \leq {|E_{m_n}|}/{n}$. This gives 
\begin{equation}\label{eqsequ3}
|C_n\cap A_n | \leq	\frac{\lvert E_{m_n}\rvert\lvert F_{k_n}\rvert}{n}= \frac{\lvert\bar F_n\rvert}{n}.
\end{equation}
Combining \eqref{eqsequ2} and \eqref{eqsequ3} we obtain
\begin{equation*}
|C_n|  =	|C_n \cap A_n| + |C_n \setminus A_n |\leq \frac{2|\bar F_n|}{n}+\frac{2|\bar F_n|}{n}\leq \frac{2|\bar F_n|}{n},
\end{equation*}	
therefore
\begin{equation}\label{cifEq2}
\frac{|\bar F_ng\setminus \bar F_n|}{|\bar F_n|}\leq \frac{2}{n}.
\end{equation}
Since $n$ was chosen arbitrarily and \eqref{cifEq2} holds for all $g\in G_n$, we have proved that $(\bar F_n)_{n\in\mathbb{N}}$ is a right F\o lner sequence of $G$.
	
By Proposition~\ref{CIF} for any sequence of natural numbers $(a_n)_{n\in\N}$ there is another sequence of natural numbers $(h_n)_{n\in\N}$ such that such that $(E_{m_{h_n}})_{n\in\N}$ is still a CIF sequence of $H$ with respect to $(H_n)_{n\in\N}$. Hence, (3) follows from (2).
\end{proof}

\subsection{\texorpdfstring{$X$}{X}-monotileable sequences}

If $G$ is a countable amenable group and $(F_n)_{n\in\N}$ a right F\o lner sequence of $G$ invariant under conjugation by any element
$g\in G$ (i.e., $gF_n = F_ng$ for all $n\in\N$ and $g\in G$), then $G$ is $\text{Inn}(G)$-monotileable.

\begin{example}\label{iii}
Here are two examples of countable amenable groups $G$  having a right F\o lner sequence $(F_n)_{n\in\N}$ that is $\Aut(G)$-monotileable.
\begin{enumerate}[(a)]
\item Take an infinite collection $\{S_n\}_{n\in\N}$ of simple finite groups such that for $n\ne m$ the only homomorphism $S_n \to S_m$ is the trivial one and let $H=\bigoplus _{n\in\N} S_n$. By our choice of the family $\{S_n\}_{n\in\N}$ one has 	$\Aut(H) = \prod_{n\in\N}\Aut(S_n)$. Consider the sequence $(\bigoplus_{i=0}^nS_i)_{n\in\N}$; it is easy to see that it is an $\Aut(H)$-monotileable right F\o lner sequence of $H$.
\item Consider the group $K=\bigoplus_{i=1}^k\Z(p^\infty_{i})$,  where $p_1, \ldots, p_k$ are pairwise distinct primes. Then for every $n\in\N$ the subgroup $K[n]:=\{k\in K: nk=0\}$ of $K$ is finite and fully invariant. Therefore the sequence $(K[n!])_{n\in\N}$ is $\Aut(K)$-monotileable.
\end{enumerate}	
\end{example}

The following lemma is the counterpart for $X$-monotileable groups of Lemma~\ref{tilesub} and has a similar proof. 

\begin{lemma}\label{stilesub}
Let $G$ be a group and $(F_n)_{n\in\N}$ an $X$-monotileable sequence with $\mathrm{id}\in S$. Then every subsequence $(F_{k_n})_{n\in\N}$ is still an $X$-monotileable sequence.
\end{lemma}

 
The following technical result is used in the proof of Theorem~\ref{expro1}. 
 
\begin{lemma}\label{exle1}
Let $H$ be a countable group and $K$ a finitely generated group with a symmetric generating set $X=\{f_1,\dots,f_m\}$. Consider a group homomorphism $\phi\colon K\rightarrow \Aut(H)$ and let $\tilde X=\{\mathrm{id},\phi(f_1),\dots,\phi(f_m)\}$. If $(E_n)_{n\in\N}$ is an $\tilde X$-monotileable sequence of $H$, then for all $f\in K$ the set $\phi(f)(E_n)$ is a monotile of $E_{n+s}$, where $s=\ell_S(f)$.
\end{lemma}
\begin{proof}
Fix $n\in\N$. We proceed by induction on the length $\ell_S$. Let $f\in K$. If $f$ has length $0$ or $1$ the statement follows by hypothesis. Suppose we have already proved the thesis for all the elements of $K$ of length $s-1$ and let $f$ such that $\ell_X(f)=s$. This means that exist an index $i\in\{1,\dots,m\}$ and an element $\bar f$ such that $f=f_i\bar f$ and $\ell_X(\bar f)=s-1$. So there exists a subset $\tilde E$ of $H$ such that
\begin{equation}\label{exeq10}
E_{n+s-1}=\bigsqcup_{\tilde e\in\tilde E} \tilde e\phi(\bar f)(E_n).
\end{equation}
By hypothesis, $\phi(f_i)(E_{n+s-1})$ is a monotile of $E_{n+s}$, therefore there exists a subset $E'$ of $H$ such that
\begin{equation}\label{exeq11}
E_{n+s}=\bigsqcup_{e'\in E'}e'\phi(f_i)(E_{n+s-1}).
\end{equation}
Combining \eqref{exeq10} and \eqref{exeq11}, we obtain
\begin{equation*}
\begin{split}
E_{n+s}&=\bigsqcup_{e'\in E'}e'\phi(f_i)\left(\bigsqcup_{\tilde e\in\tilde E} \tilde e\phi(\bar f)(E_n)\right)\\&=\bigsqcup_{e'\in E'}e'\left(\bigsqcup_{\tilde e\in\tilde E} \phi(f_i)(\tilde e)\phi(f_i)(\phi(\bar f)(E_n))\right) \\ &=\bigsqcup_{e'\in E'}e'\left(\bigsqcup_{\phi(f_i)(\tilde e)\in\phi(f_i)(\tilde E)} \phi(f_i)(\tilde e)\phi(f)(E_n)\right).
\end{split}
\end{equation*}
Let $E''= E'\phi(f_i)(\tilde E)$. By Lemma~\ref{disun}, $E_{n+s}=\bigsqcup_{e''\in E''}e''\phi(f)(E_n)$.
\end{proof}

\subsection{Proof of the Extension Theorem}

Consider two countable groups $G$ and $K$ and a surjective homomorphism $\pi\colon G\rightarrow K$. Fix a section $\sigma\colon K\rightarrow G$ for $\pi$, such that $\sigma(1_K)=1_G$. Let $(F_n)_{n\in\N}$ be a \tileable\ sequence of $K$ and let $(K_n)_{n\in\N}$ be a tiling sequence associated to $(F_n)_{n\in\N}$. We recall that any element  $f\in F_n$ can be written in a unique way as $f=\prod_{j=1}^nk_j$, where $k_j\in K_{n+1-j}$ for $1\leq j\leq n$. For all $n\in\N$ we define 
\begin{equation*}
\sigma_n\colon F_n\rightarrow G\ \text{ as }\ \sigma_n(f)=\prod_{j=1}^n\sigma(k_j)\ \text{ where }f=\prod_{j=1}^nk_j.
\end{equation*}
Since $F_n\subseteq F_{n+1}$ for all $n\in\N$, we obtain that $\sigma_{n+1}\rest_{F_n}=\sigma_n$ for all $n\in\N$. Then the map $\bar \sigma\colon \bigcup_{n\in\N} F_n\rightarrow G$, given by  $\bar \sigma(f)=\sigma_n(f)$ for $n\in\N$ such that $f\in F_n$, is well defined. 

It is straightforward to verify that $\bar\sigma(1_K)=\sigma_0(1_K)=\sigma(1_K)=1_G$ and that, for every $f\in \bigcup_{n\in\N}F_n$, $\pi(\bar\sigma(f))=f$, that is, 
$\pi\circ\bar\sigma=\mathrm{id}_K\rest_{\bigcup_{n\in\N}F_n}$; in particular $\bar \sigma$ is injective.
Hence, we can extend the map $\bar \sigma$ to a section $\tilde\sigma\colon K\rightarrow G$ for $\pi$ and we call it the \emph{section associated} to $(F_n)_{n\in\N}$ and $\sigma$. Clearly, this $\tilde\sigma$ need not be unique. In case $\tilde\sigma = \sigma$, we simply say that $\sigma$
is associated to $(F_n)_{n\in\N}$. 

\begin{lemma}\label{tile1}
Let  $G$ and $K$ be countable groups, $\pi\colon G\rightarrow K$ a surjective homomorphism and $\sigma\colon K\rightarrow G$ a section  for $\pi$ with  $\sigma(1_K)=1_G$. For a \tileable\ sequence $(F_n)_{n\in\N}$  of $K$ with associated tiling sequence 
$(K_n)_{n\in\N}$ consider also the section $\tilde \sigma$ associated to $(F_n)_{n\in\N}$ and $\sigma$.
Then $(\tilde\sigma(F_n))_{n\in\N}$ is a \tileable\ sequence of $G$ with associate tiling sequence $(\sigma(K_n))_{n\in\N}$.
\end{lemma}
\begin{proof}
To verify that $(\tilde\sigma(F_n))_{n\in \N}$ is a \tileable\ sequence we have to prove that 
\begin{equation}\label{abcd}
\sigma(K_{n+1})^{-1}\sigma(K_{n+1}) \cap \tilde\sigma(F_n)\tilde\sigma(F_n)^{-1} = \{1\},
\end{equation}
for every $n\in\N$.
Fix $n\in\N$ and $g\in\sigma(K_{n+1})^{-1}\sigma(K_{n+1}) \cap \tilde\sigma(F_n)\tilde\sigma(F_n)^{-1}$. So there are $k_1,k_2\in K_{n+1}$ and $f_1,f_2\in \tilde\sigma(F_n)$ such that $g=f_1 f_2^{-1}$ and $\sigma(k_1)f_1=\sigma(k_2)f_2$. Thus, applying $\pi$,
\begin{equation*}
k_1\pi(f_1)=\pi(\sigma(k_1)f_1)=\pi(\sigma(k_2)f_2)=k_2\pi(f_2).
\end{equation*}
Since $(F_n)_{n\in\N}$ is a \tileable\ sequence of $K$, from $k_1\pi(f_1)=k_2\pi(f_2)$ we conclude that $k_1=k_2$ and $\pi(f_1)=\pi(f_2)$. This implies $\sigma(k_1)=\sigma(k_2)$ and $f_1=f_2$. Therefore, $g=1$ and then \eqref{abcd} holds.
\end{proof}

We are now in position to prove Theorem~\ref{monosequ1}, that is the Extension Theorem stated in the introduction. Actually, here we prove the more technical claim below, from which the theorem follows immediately.

\begin{claim}\label{monosequ}
Suppose that $0 \rightarrow H \xrightarrow{\iota} G \xrightarrow{\pi}K \rightarrow 0$ is a short exact sequence of countable groups, where $K,H$ are \tileable. Fix a section $\sigma\colon K\rightarrow G$ for $\pi$ such that $\sigma(1_K)=1_G$, a \tileable\ F\o lner sequence $(F_n)_{n\in\N}$ of $K$ and a \tileable \ F\o lner sequence $(E_n)_{n\in\N}$ of $H$. Let $(K_n)_{n\in\N}$ be a tiling sequence associated to $(F_n)_{n\in\N}$ and $\tilde\sigma$ a section associated to $(F_n)_{n\in\N}$ and $\sigma$. For convenience, for every $n\in\N$, we set $F_n'= \tilde\sigma(F_n)$ and $K_n'=\sigma(K_n)$.  If one of the following two conditions holds:
\begin{enumerate}[(a)]
\item  $(E_n)_{n\in\N}$ is $\mathrm{Inn}(G)$-monotileable,
\item  $K'_n=\sigma(K_n) \subseteq c_G(H)$ for all $n\in\N$,
\end{enumerate}
then there exist a \tileable\ sequence $(F^{\#}_n)_{n\in\N}$ of $G$ and a strictly increasing sequence $(m_n)_{n\in\N}$ in $\N$ such that:
\begin{enumerate}[(1)]
\item $(\pi(F^{\#}_n))_{n\in\N}$ is a \tileable\ right F\o lner sequence of $K$;
\item $(\bar F_n)_{n\in\N}=(E_{m_n}F^{\#}_n)_{n\in\N}$ is a \tileable\ right F\o lner sequence of $G$;
\item for all $n\in\N$, $\bar F_n=\bigsqcup_{f\in F^{\#}_n}E_{m_n}f$.
\end{enumerate}
\end{claim}
\begin{proof}
By Theorem~\ref{teofolner} applied to the section $\tilde\sigma$, there exist two strictly increasing sequences of natural numbers $(m_n)_{n\in\N}$ and $(k_n)_{n\in\N}$ such that the sequence $(\bar F_n)_{n\in\N}$ given by $\bar F_n=E_{m_n}F'_{k_n}$ for all $n\in\N$ is a right F\o lner sequence of $G$. So, for all $n\in\N$ let $F^{\#}_n= F'_{k_n}$.  Moreover, Lemma~\ref{tidisun} (with $j=m_n$ and $i=n$) gives $\bar F_n=\bigsqcup_{f^{\#}\in F^{\#}_n}E_{m_n}f^{\#}$ for all $n\in\N$, and so also (3) holds.
 Since, for every $n\in\N$, $\pi(F^{\#}_n)=\pi(F'_{k_n})=\pi(\tilde\sigma(F_{k_n}))=F_{k_n}$, the condition in (1) is satisfied by Lemma~\ref{tilesub}.

It remains to prove only that $(\bar F_n)_{n\in\N}$ is \tileable, so that also (2) is satisfied. The sequence $(F'_n)_{n\in\N}$ is \tileable\ by Lemma~\ref{tile1} 
and so also $(F'_{k_n})_{n\in\N}$ is \tileable\ by Lemma~\ref{tilesub}, that is, for every $n\in\N$ there exists a finite subset $\bar K_n$ of $G$ such that 
\begin{equation}\label{tieq4}
F_n^{\#}=F'_{k_n}=\bigsqcup_{f\in\bar K_n}fF'_{k_{n-1}}.
\end{equation} 
If (a) holds, then $(E_{m_n})_{n\in\N}$ is $\mathrm{Inn}(G)$-monotileable by Lemma~\ref{stilesub} and so, for every $n\in\N$ and $f\in\bar K_n$, the subset $fE_{m_{n-1}}f^{-1}$ is a monotile of $E_{m_n}$. If (b) holds, we immediately conclude that $\bar K_n\subseteq c_G(H)$, and so, for every $n\in\N$ and $f\in\bar K_n$,  $E_{m_{n-1}}=fE_{m_{n-1}}f^{-1}$ is a monotile of $E_{m_n}$. Hence, in both cases, for every $n\in\N$ there exists a finite subset $\bar E_{f,n}$ of $H$ such that
\begin{equation}\label{tieq5}
E_{m_n}=\bigsqcup_{e\in\bar E_{f,n}}\bar efE_{m_{n-1}}f^{-1}.
\end{equation} 
For every $n\in\N$ let
$\bar G_n=\bigsqcup_{f\in \bar K_n}\bar E_{f,n}f.$
Fixed $n\in\N_+$, \eqref{tieq4} and Lemma~\ref{tidisun}(b) yield
\begin{equation}\label{eqsist}
\bar F_n= E_{m_n}F'_{k_n}=E_{m_n}\bigsqcup_{f\in\bar K_n}fF'_{k_{n-1}}=\bigsqcup_{f\in\bar K_n}E_{m_n}fF'_{k_{n-1}},
\end{equation}
and by \eqref{tieq5}, \eqref{eqsist} and Lemma~\ref{tidisun}(b), we have 
\begin{equation*}\begin{split}
\bar F_n&=\bigsqcup_{f\in\bar K_n}E_{m_n}fF'_{k_{n-1}}=\bigsqcup_{f\in\bar K_n}\left(\left(\bigsqcup_{e\in \bar E_{f,n}}efE_{m_{n-1}}f^{-1}\right)fF'_{k_{n-1}}\right)=\\
&=\bigsqcup_{f\in\bar K_n}\bigsqcup_{e\in \bar E_{f,n}}efE_{m_{n-1}}F'_{k_{n-1}}= \bigsqcup_{f\in\bar K_n}\bigsqcup_{e\in \bar E_{f,n}}ef\bar F_{n-1}=\bigsqcup_{g\in\bar G_n}g\bar F_{n-1}.
\end{split}	\end{equation*}
Then $(\bar F_n)_{n\in\N}$ is \tileable.
\end{proof}

The following corollary about the stability of $\M$ under central extensions coincides with \cite[Lemma 6]{CC}.

\begin{corollary}\label{ticor3}
Consider two countable groups $G$ and $K$. Let $G$ be a central extension of $K$, and let $H\leq Z(G)$ such that
$0 \rightarrow H \rightarrow G \rightarrow K \rightarrow 0$ is a short exact sequence. If $H$ and $K$ are \tileable, then also $G$ is \tileable.
\end{corollary}
\begin{proof}
Since $H$ is abelian, every \tileable\ sequence of $H$  is trivially $\mathrm{Inn}(G)$-monotileable. Therefore Claim~\ref{monosequ} applies.
\end{proof}

\begin{example}
The Heisenberg group $H_3(\mathbb{Z})$ is the group of $3\times 3$ upper unitriangular matrices in $M_3(\mathbb{Z})$. 
Since $H_3(\mathbb{Z})$ is a central extension of $\Z$ by $\mathbb{Z}^2$, so $H_3(\mathbb{Z})$ is locally monotileable by Corollary~\ref{ticor3}.  
\end{example}

If in Claim~\ref{monosequ} $H$ is finite, we can consider the \tileable\ F\o lner sequence $(E_n)_{n\in\N}$ given by $E_0=\{1_G\}$ and $E_n=H$ for all $n>0$. Since $H=\ker\pi$, we have that $H$ is normal in $G$. Therefore the sequence $(E_n)_{n\in\N}$ is invariant under conjugation by any element $g\in G$ and so it is $\mathrm{Inn}(G)$-monotileable. Thus we can apply Claim~\ref{monosequ} and  we obtain the following result.

\begin{corollary}
Consider two countable groups $G$ and $K$. Let $\pi\colon G\rightarrow K$ be a surjective homomorphism with $\ker \pi$ finite. If $K$ is \tileable\ then also $G$ is \tileable. 
\end{corollary}

\subsection{Countable virtually nilpotent groups are locally monotileable}

In this section we use the Extension Theorem (actually Claim~\ref{monosequ}) to prove first that a countable group with a locally monotileable normal subgroup of finite index is necessarily \tileable, and then that all countable abelian groups are \tileable. These results together give that all countable virtually nilpotent groups are \tileable. Note that the virtually nilpotent finitely generated groups are precisely those of polynomial growth by the celebrated Gromov's Theorem.

\begin{proposition}\label{ticor1}
If $G$ is a countable group having a normal subgroup of finite index $H$ which is \tileable, then so is $G$. 
\end{proposition}
\begin{proof}
Let $(E_n)_{n\in\N}$ be a \tileable\ right F\o lner sequence of $H$. Let $K =G/H$, let $\pi:G\rightarrow K$ be the canonical projection and $\iota:H\to G$ be the inclusion of $H$ in $G$. Fix a section $\sigma\colon K\rightarrow G$ for $\pi$ such that $\sigma(1_K)=1_G$, and let $\sigma(K)=R$. 
Since $K$ is finite, consider the \tileable\ F\o lner sequence $(F_n)_{n\in\N}$ given by $F_0=\{1_K\}$ and $F_n=K$ for all $n\in\N_+$. By Theorem~\ref{teofolner} there is an increasing sequence of natural numbers $(m_n)_{n\in\N}$ such that the sequence $(\bar F_n)_{n\in\N}$, given by $\bar F_0=\{1\}$ and $\bar F_n=E_{m_n}R$ for all $n\in\N_+$, is a right F\o lner sequence of $G$ (note that $R$ is finite and so also $\bar F_n$ is finite for every $n\in\N$). 
	
It remains to prove that $(\bar F_n)_{n\in\N}$ is \tileable. Clearly $\bar F_0$ is a monotile of $\bar F_1$, so we can suppose $n>1$. The sequence $(E_n)_{n\in\N}$ is \tileable\ by hypothesis so by Lemma~\ref{tilesub} also $(E_{m_n})_{n\in\N}$ is \tileable. Therefore for every $n>1$ there is a finite subset $\bar E_n$ of $H$ such that 
\begin{equation}\label{tieq31}
E_{m_{n}}=\bigsqcup_{\bar e\in\bar E_{n}}\bar eE_{m_{n-1}}.
\end{equation} 
For $\bar e_1,\bar e_2\in\bar E_n$ with $\bar e_1\neq\bar e_2$, we have $\bar e_1E_{m_{n-1}}\cap \bar e_2E_{m_{n-1}}=\emptyset$. Therefore, since $R$ is a set of right coset representatives of $H$ in $G$ and $E_{m_n}\subseteq H$,
\begin{equation}\label{tieq32}
\bar e_1E_{m_{n-1}}R\cap\bar e_2E_{m_{n-1}}R=\emptyset.
\end{equation}
By \eqref{tieq31} and \eqref{tieq32} we conclude that, for all $n\in\N$, $\bar F_n=E_{m_n}R=\bigsqcup_{\bar e\in\bar E_n}\bar eE_{m_{n-1}}R=\bigsqcup_{\bar e\in\bar E_n}\bar e\bar F_{n-1}.$
\end{proof}

\begin{example}
Consider the group $G = \Z \rtimes \Z_2$, where $\Z_2$ acts on $\Z$ as the automorphism $t(x)=-x$. The group $G$ is the semidirect product between $\Z$ and $\Z_2$ therefore
\(0\rightarrow \Z\rightarrow G\rightarrow \Z_2\rightarrow0\) is a short exact sequence. We know that $\Z$ and $\Z_2$ are both \tileable\ groups, so by Proposition~\ref{ticor1} we have that also $G$ is \tileable.
%
%
\end{example}

We apply twice the following easy observation in the proof of Theorem~\ref{tileable}.

\begin{lemma}\label{tilecard}
Let $G$ be a group and $A, B\in\mathcal{P}_{fin}(G)$. If $g\in G$, then
$$|(AgB)\setminus(AB)|\leq|(Ag)\setminus A|\,|B|\quad\text{and}\quad |(ABg)\setminus(AB)|\leq|(Bg)\setminus B|\,|A|.$$
\end{lemma}
\begin{proof}
Fix $g\in G$ and let $a\in A$ and $b\in B$ such that $agb\notin AB$. This clearly implies that $ag\notin A$. So, $(AgB)\setminus(AB)\subseteq (Ag\setminus A)B$, and it is straightforward to deduce the first inequality. 
The second inequality can be proved analogously.
\end{proof}

\begin{theorem}\label{tileable}
Assume that the group $G$ is increasing union of its subgroups $G_0\leq G_1\leq G_2\leq\dots$, and that:
\begin{enumerate}[(1)]
\item every $G_n$ is \tileable,
\item every quotient $G_{n+1}/G_n$ is \tileable,
\item every quotient $G_{n+1}/G_n$ admits a section $\sigma_n:G_{n+1}/G_n \to G_{n+1}$ such that $\sigma_n(1)=1$ and $\sigma_n(G_{n+1}/G_n) \leq c_{G_{n+1}}(G_n)$.
\end{enumerate}
Then $G$ is \tileable.
\end{theorem}
\begin{proof}
We build recursively a \tileable\ right F\o lner sequence $(E_{n,j})_{j\in\mathbb{N}}$ of $G_n$ for every $n\in\N$.
Let $(E_{0,j})_{j\in\mathbb{N}}$ be a \tileable\ right F\o lner sequence of $G_0$. Fix $n\in\mathbb{N}$ and suppose we have already defined a \tileable\ right F\o lner sequence $(E_{n,j})_{j\in\mathbb{N}}$ of $G_n$. Consider the exact sequence
\begin{equation}\label{tieqexse}
0\rightarrow G_n\rightarrow G_{n+1}\rightarrow {G_{n+1}}/{G_n}\rightarrow 0.
\end{equation} 
Since $\sigma_n(G_{n+1}/G_n) \leq c_{G_{n+1}}(G_n)$, by Claim~\ref{monosequ} there exist a \tileable\ sequence $(F^{\#}_{n+1,j})_{j\in\N}$ of $G_{n+1}$ and a strictly increasing sequence of natural numbers $(m_{n+1,j})_{j\in\N}$ with $m_{n+1,j}\geq j$ for all $j\in\N$, such that the sequence $(E_{n+1,j})_{j\in\N}$, defined letting, for every $j\in\N$,
\begin{equation*}
E_{n+1,j}=E_{n,m_{n+1,j}}F^{\#}_{n+1,j},
\end{equation*}
is a \tileable\ right F\o lner sequence of $G_{n+1}$. Since $F_{n+1,j}^{\#}\subseteq \sigma_n(G_{n+1}/G_n) \leq c_{G_{n+1}}(G_n)$, then
\begin{equation}\label{tieqabe0}
E_{n+1,j}=\bigsqcup_{f\in F^{\#}_{n+1,j}}fE_{n,m_{n+1,j}}.
\end{equation}	
It remains to prove that the diagonal sequence $(E_{n,n})_{n\in\mathbb{N}}$ is a \tileable\ right F\o lner sequence of $G$.

\begin{claim}\label{claim2}
$(E_{n,n})_{n\in\mathbb{N}}$ is a \tileable\ sequence of $G$.
\end{claim}
For every $n\in\N$, \eqref{tieqabe0} yields
\begin{equation}\label{ult1}
E_{n+1,n+1}=\bigsqcup_{f\in F^{\#}_{n+1,n+1}}fE_{n,m_{n+1,n+1}}.
\end{equation}
Let $(K_{n,j})_{n\in\N}$ be a tiling sequence associated to $(E_{n,j})_{n\in\N}$. Since for every $n\in\N$ and $j\in\N$, we have $m_{n+1,j}> j$, the product $M_{n,j}=K_{n,m_{n+1,j}}\ldots K_{n,j+1}$ makes sense and 
\begin{equation}\label{eqnew}
E_{n,m_{n+1,j}}=M_{n,j}E_{n,j}.
\end{equation}

Fix $n\in\N$ and $j=n+1$. Then $M_{n,n+1} = K_{n,m_{n+1,n+1}}\ldots K_{n,n+2}$, and letting $M=M_{n,n+1}K_{n,n+1}$,
\begin{equation}\label{ult2}
E_{n,m_{n+1,n+1}}=M_{n,n+1}K_{n,n+1}E_{n,n}= ME_{n,n}=\bigsqcup_{m\in M} mE_{n,n}.
\end{equation}
Combining \eqref{ult1} and \eqref{ult2}, we get
\begin{equation*}
E_{n+1,n+1}=\bigsqcup_{f\in F^{\#}_{n+1,n+1}}f\left(\bigsqcup_{m\in M}mE_{n,n}\right).
\end{equation*}	
For $A_{n+1}=F^{\#}_{n+1,n+1}M$, by Lemma~\ref{disun} we obtain that $E_{n+1,n+1}=\bigsqcup_{a\in A_{n+1}}aE_{n,n}$. This proves Claim~\ref{claim2}.

\begin{claim}\label{claim1}
$(E_{n,n})_{n\in\mathbb{N}}$ is a right F\o lner sequence of $G$.
\end{claim}
Take $g\in G$. Since $G=\bigcup_{n\in\mathbb{N}}G_n$, there exists $t\in \N$ such that $g\in G_t$. Pick $n\in\N$ with $n\geq t$. Then $g\in G_n$, so $gF^{\#}_{n+1,j}=F^{\#}_{n+1,j}g$ for every $j\in\N$. Moreover, for every $j\in\N$,
\begin{equation}\label{eqabe1}
\begin{split}
\frac{|E_{n+1,j}g\setminus E_{n+1,j}|}{|E_{n+1,j}|}&=	\frac{|E_{n,m_{n+1,j}}F^{\#}_{n+1,j}g\setminus (E_{n,m_{n+1,j}}F^{\#}_{n+1,j})|}{|E_{n,m_{n+1,j}}F^{\#}_{n+1,j}|}=\frac{|E_{n,m_{n+1,j}}gF^{\#}_{n+1,j}\setminus (E_{n,m_{n+1,j}}F^{\#}_{n+1,j})|}{|E_{n,m_{n+1,j}}F^{\#}_{n+1,j}|}\leq\\
&\leq\frac{|E_{n,m_{n+1,j}}g\setminus E_{n,m_{n+1,j}}||F^{\#}_{n+1,j}|}{|E_{n,m_{n+1,j}}||F^{\#}_{n+1,j}|}=\frac{|E_{n,m_{n+1,j}}g\setminus E_{n,m_{n+1,j}}|}{|E_{n,m_{n+1,j}}|},
\end{split}\end{equation}
where the inequality holds by Lemma~\ref{tilecard}.
By Remark~\ref{tire1} and by \eqref{eqnew} we obtain that
\begin{equation}\label{tieq9}
|E_{n,m_{n+1,j}}|=|M_{n,j}E_{n,j}|=|K_{n,m_{n+1,j}}\ldots K_{n,j+1}E_{n,j}|=|K_{n,m_{n+1,j}}|\ldots|K_{n,j+1}||E_{n,j}|=|M_{n,j}||E_{n,j}|.\end{equation}
Now \eqref{tieq9} and Lemma~\ref{tilecard} give
\begin{equation}\label{eqabe2}
\frac{|E_{n,m_{n+1,j}}g\setminus E_{n,m_{n+1,j}}|}{|E_{n,m_{n+1,j}}|}=\frac{|M_{n,j}E_{n,j}g\setminus (M_{n,j}E_{n,j})|}{|M_{n,j}E_{n,j}|}\leq\frac{|E_{n,j}g\setminus E_{n,j}||M_{n,j}|}{|E_{n,j}||M_{n,j}|}=\frac{|E_{n,j}g\setminus E_{n,j}|}{|E_{n,j}|},
\end{equation}
and so \eqref{eqabe1} and \eqref{eqabe2} yield
\begin{equation*}\label{eqabe3}
\frac{|E_{n+1,j}g\setminus E_{n+1,j}|}{|E_{n+1,j}|}\leq\frac{|E_{n,j}g\setminus E_{n,j}|}{|E_{n,j}|}.
\end{equation*}
Using this inequality for $n=t$, by induction we get that, for every $n\geq t$,
\begin{equation}\label{eqabe4}
\frac{|E_{n,j}g\setminus E_{n,j}|}{|E_{n,j}|}\leq\frac{|E_{t,j}g\setminus E_{t,j}|}{|E_{t,j}|}.
\end{equation} 
By the choice of the sequence $(E_{t,j})_{j\in\N}$, 
\begin{equation}\label{eqabe5}
\lim_{j\to\infty}\frac{\left| E_{t,j}g\setminus E_{t,j}\right|}{\left| E_{t,j}\right|}=0.
\end{equation}
By \eqref{eqabe4} and \eqref{eqabe5} we conclude that
\begin{equation*}
\lim_{n\to\infty}\frac{| E_{ n,n}g\setminus E_{n,n}|}{|E_{ n,n}|}=\lim_{j\to\infty}\frac{|E_{t+j,t+j}g\setminus E_{t+j,t+j}|}{|E_{t+j,t+j}|}
\leq\lim_{j\to\infty}\frac{| E_{t,t+j}g\setminus E_{t,t+j}|}{|E_{t,t +j}|}=\lim_{j\to\infty}\frac{|E_{t,j}g\setminus E_{t,j}|}{|E_{t,j}|}=0.
\end{equation*}
This proves Claim~\ref{claim1}, and so concludes the proof of the theorem.
\end{proof}

Now Corollary~\ref{monofingen} can be obtained as a consequence of Theorem~\ref{tileable}.
Moreover, by Theorem~\ref{tileable} we have that $\Q \in \M$ (for an alternative proof see Theorem~\ref{tipro3}) and the following more general results.

\begin{corollary}\label{abelian}
If $G$ is a countable abelian group, then $G\in\M$.
\end{corollary}
\begin{proof}
By hypothesis $G$ is increasing union of and increasing chain $\{G_n\colon n\in\N\}$ of finitely generated subgroups; moreover, $G_{n+1}/G_n$ is finitely generated for every $n\in\N$. By Corollary~\ref{monofingen}, for every $n\in\N$, $G_n$ and $G_{n+1}/G_n$ are \tileable. Hence $G$ is \tileable\ by Theorem~\ref{tileable}. 
\end{proof}

\begin{corollary}\label{hypercentral}
If $G$ is a countable hypercentral group of length $<\omega^2$, then $G\in\M$.
\end{corollary}
\begin{proof}
 By hypothesis $G=Z_\alpha(G)$ for some countable ordinal $\alpha<\omega^2$. We prove by induction that $Z_\kappa(G)$ is locally monotileable for every ordinal $\kappa<\omega^2$. Indeed, $Z_0(G)=\{1\}$ is locally monotileable. Moreover, if $\kappa=\beta+1$ for some ordinal $\beta$, since $Z_\beta(G)$ is locally monotileable by the inductive hypothesis and $Z_\kappa(G)/Z_\beta(G)$ is locally monotileable by Corollary~\ref{abelian}, $Z_\kappa(G)$ is locally monotileable by Corollary~\ref{ticor3}. If $\kappa$ is a limit ordinal, then $\kappa=m\omega$ for some $m\in\N_+$, and so $Z_\kappa(G)$ is increasing union of its subgroups $\{Z_{(m-1)\omega+n}(G)\colon n\in\N\}$. By the inductive hypothesis and by Corollary~\ref{abelian}, those subgroups satisfy the hypotheses of Theorem~\ref{tileable}, so $Z_\kappa(G)$ is \tileable.
\end{proof}

 Corollary~\ref{hypercentral} and Proposition~\ref{ticor1} give Theorem~\ref{vhy}.



\section{Applications of the Extension Theorem}

\subsection{Extensions of $\Q$ by $\Z$}

The following technical lemma is needed in the proof of Theorem~\ref{expro1}.

\begin{lemma}\label{newle}
Let $H, K$ be countable groups, $\phi\colon K\rightarrow \Aut(H)$ a group homomorphism and $(E_n)_{n\in\N}$ a \tileable\ sequence of $H$. Consider also a \tileable\ sequence $(F_n)_{n\in\N}$ of $K$ and let $(K_n)_{n\in\N}$ be an associated tiling sequence. Let $(\bar F_n)_{n\in\N}$ be the sequence in $\Pf(H\times K)$ given by $\bar F_n= E_n\times F_n$. If for all $n\in\N$ and $k\in K_{n+1}$, $\phi(k)(E_n)$ is a monotile of $E_{n+1}$,
then $(\bar F_n)_{n\in\N}$ is a \tileable\ sequence of $H\rtimes_{\phi}K$.
\end{lemma}
\begin{proof}
Fix $n\in\N$.
By hypothesis $\phi(k)(E_{{n}})$ is a monotile of $E_{n+1}$ for all $k\in K_{n+1}$. Therefore for all $k\in K_{n+1}$, there is $\tilde E_{n+1,k}$ such that 
$E_{{n+1}}=\bigsqcup_{\tilde e\in\tilde E_{n+1,k}}\tilde e\phi(k)(E_{n})$. 
Consider 
$\bar K_{n+1}=\bigsqcup_{k\in K_{n+1}} \tilde E_{n+1,k}\times \{k\}.$
Then 
\[\bar F_{n+1}=E_{{n+1}}\times F_{n+1}=\bigsqcup_{k\in K_{n+1}}E_{{n+1}}\times (kF_n)=\bigsqcup_{\bar k\in\bar K_{n+1}} \bar k(E_{n}\times F_n)=\bigsqcup_{\bar k\in \bar K_{n+1}} \bar k\bar F_n.\qedhere\]
\end{proof}

\begin{theorem}\label{expro1}
Let $H$ be a countable group and $K$ a \tileable\ finitely generated group with a symmetric generating set $X=\{f_1,\ldots,f_m\}$. Consider a group homomorphism $\phi\colon K\rightarrow \Aut(H)$ and let $\tilde X=\{\mathrm{id},\phi(f_1),\ldots,\phi(f_m)\}$. If $H$ is $\tilde X$-monotileable then the group $G= H\rtimes_{\phi}K$ is \tileable.
\end{theorem}
\begin{proof}
Let $(E_n)_{n\in\N}$ be an $\tilde X$-monotileable right F\o lner sequence of $H$ and $(F_n)_{n\in\N}$ a \tileable\ right F\o lner sequence of $K$.
By Theorem~\ref{teofolner} there exist two sequences $(m_n)_{n\in\N}$ and $(k_n)_{n\in\N}$ in $\N$ such that the sequence $(\bar F_n)_{n\in\N}$, given by $\bar F_n= E_{m_n}\times F_n$ for every $n\in\N$, is a right F\o lner sequence of $H\rtimes_{\phi} K$.
	
Consider a tiling sequence $(K_n)_{n\in\N}$ associated to $(F_{k_n})_{n\in\N}$.
For all $n\in\N$ let $a_n=\max\{\ell_S(k): k\in K_{n+1}\}$.
By Theorem~\ref{teofolner} there exists a strictly increasing sequence $(h_n)_{n\in\N}$ in $\N$ such that:
\begin{enumerate}[(1)]
\item $m_{h_{n+1}}>m_{h_n}+a_n$ for every $n\in\N$;
\item the sequence $(F^*_n)_{n\in\N}$, given by $F^*_n= E_{m_{h_n}}\times F_n$ for every $n\in\N$, is a right F\o lner sequence of $H\rtimes_{\phi} K$.
\end{enumerate}	
By Lemma~\ref{stilesub} the subsequence $(E_{m_{h_n}})_{n\in\N}$ is still an $\tilde X$-monotileable sequence of $H$. 
By Lemma~\ref{exle1} we know that $\phi(k)(E_{m_{h_n}})$ is a monotile of $E_{m_{h_n}+\ell_S(k)}$ for all $k\in K_{n+1}$. On the other hand, since $m_{h_n}+\ell_S(k)\leq m_{h_{n+1}}$, by Lemma~\ref{tilesub} $E_{m_{h_n}+\ell_S(k)}$ is a monotile of $E_{m_{h_{n+1}}}$. Finally by Lemma~\ref{disun} we have that $\phi(k)(E_{m_{h_n}})$ is a monotile of $E_{m_{h_{n+1}}}$ for all $k\in K_{n+1}$.
	
The sequences $(E_{m_{h_n}})_{n\in\N}$, $(F_{k_n})_{n\in\N}$ and $(F^*_n)_{n\in\N}$ satisfy the hypothesis of Lemma~\ref{newle}, so we apply it to conclude.
\end{proof}		

The following is an immediate consequence of Theorem~\ref{expro1} (take $X=\{1,-1\}$ as a finite generating set of $\Z$).

\begin{corollary}\label{exco1}
Let $H$ be a countable \tileable\  group and $\phi\colon H\rightarrow H$ an automorphism. If $H$ is $\phi$-monotileable, then $H\rtimes_{\phi}\Z$ is \tileable.
\end{corollary}

Now we prove Theorem~\ref{tipro3} stating that for every automorphism $\phi$ of $\Q$, the group $\Q \rtimes_{\phi}\Z$ is \tileable. First we need the following folklore lemma; recall that an a-sequence $(a_n)_{n\in\N}$ is \emph{geometric} if for every $k\in\N$ there exists $\bar n\in\N$ with $k\mid a_{\bar n}$.

\begin{lemma}\label{geotil}
Let $(a_n)_{n\in\N}$ be a geometric a-sequence.
\begin{enumerate}[(a)]
\item For every $q\in\Q$ there is a minimum $n_q\in \N$ such that $q\in\langle\frac{1}{a_{n_q}}\rangle$.
\item For $q$ and $n_q$ as in item (a), there exist unique $k_1,\dots,k_{n_q}\in\N$ and $k_0\in\Z$ such that $q=\sum_{i=0}^{n_q}\frac{k_i}{a_i}$ and $0\leq k_i<q_i=a_i/a_{i-1}$ for all $i\in\{1,\dots,n_q\}$.
\end{enumerate}
\end{lemma}
\begin{proof} 
(a) Fix $q=\frac{s}{t}\in\Q$, where $(s,t)=1$ and $0<t=p_1^{k_1}\dots p_m^{k_m}$. Since $(a_n)_{n\in\N}$ is a geometric a-sequence, 
for every $i = 1,2,\ldots, m$ there exists a minimal $n_i\in\N$ such that $p_i^{k_i}\mid a_{n_i}$. Let $n_q=\max\{n_i: 1\leq i\leq m\}$, then $q\in\langle 1/a_{n_q}\rangle$ and $n_q$ is  minimal with this property.
		
(b) We proceed by induction on $n_q$. If $n_q=0$ then $q\in\Z$ and the statement is known to be true. Fix $n\in\N$ and suppose that we already proved the statement for all the $q\in\Q$ such that $n_q\leq n$. Fix $q=\frac{s}{a_{n+1}}\in \langle\frac{1}{a_{n+1}}\rangle$ such that $n_q=n+1$, then there are unique $0\leq k_{n+1}<q_{n+1}$ and $s'\in\Z$ such that $s=k_{n+1}+s'q_{n+1}$ and so, by inductive hypothesis,
\begin{equation*}\label{tieq1.9}
q=\frac{k_{n+1}}{a_{n+1}}+\frac{s'}{a_n}=\frac{k_{n+1}}{a_{n+1}}+\sum_{i=0}^{n}\frac{k_i}{a_i}=\sum_{i=0}^{n+1}\frac{k_i}{a_i}.\qedhere
\end{equation*}
\end{proof}

In order to prove Theorem~\ref{tipro3}, we start with the following fact.

\begin{claim}\label{monoqc}
Let $(a_n)_{n\in\N}$ be a geometric a-sequence and $(c_n)_{n\in\N}$ an a-sequence, then the sequence $(F_n)_{n\in\N}$, with
$F_0=\{0\}$ and $F_n=\langle\frac{1}{a_n}\rangle\cap[0,c_n)$ for every $n\in\N_+$, is F\o lner sequence of $\Q$.
\end{claim}
\begin{proof}
Fix $0\leq q=\frac{s}{t}\in\Q$. There exists $n'$ such that $q<c_{n'}$ and by Lemma~\ref{geotil} there is $n_q$ such that $\frac{1}{t}\in F_{n_q}$. Let $\bar n=\max\{n_q, n'\}$, then $q\in F_{\bar n}$.
Note that $F_{\bar n+k}$ can be covered by $c_{\bar n+k}/c_{\bar n}$ disjoint sets of the form 
\begin{equation*}
F_{\bar n+k}\cap[lc_{\bar n},(l+1)c_{\bar n})=\left\langle\frac{1}{a_{\bar n+k}}\right\rangle\cap[lc_{\bar n},(l+1)c_{\bar n}),
\end{equation*}
where $l\in\{0,c_{\bar n+k}/c_{\bar n}-1\}$. In each of these sets there are exactly $\left |F_{\bar n+k}\right|c_{\bar n}/c_{\bar n+k}$ elements. If we translate $F_{\bar n+k}$ by $c_{\bar n}$, all these sets, except the last one, are shifted exactly in the next one. So,
\begin{equation*}
(c_{\bar n}+F_{\bar n+k})\setminus F_{\bar n+k}=\left(\left\langle\frac{1}{a_{\bar n+k}}\right\rangle\cap[(c_{k+\bar n}-c_{\bar n},c_{k+\bar n})\right)+c_{\bar n},
\end{equation*}
and then 
\begin{equation}\label{eqqq2}
\left|(c_{\bar n}+F_{\bar n+k})\setminus F_{\bar n+k}\right|=\frac{\left|F_{\bar n+k}\right|c_{\bar n}}{c_{\bar n+k}}.
\end{equation}
Moreover $q<c_{\bar n}$, so $q\in\langle\frac{1}{a_{\bar n+k}}\rangle$ for all $k\in\N$. Thus, $(q+F_{\bar n+k})\cup F_{\bar n+k}\leq (c_{\bar n}+F_{\bar n+k})\cup F_{\bar n+k}$,
and this implies
\begin{equation}\label{eqqq1}
(q+F_{\bar n+k})\setminus F_{\bar n+k}\subseteq c_{\bar n}+F_{\bar n+k}\setminus F_{\bar n+k}.
\end{equation}
By \eqref{eqqq2} and \eqref{eqqq1} we obtain that
\begin{equation*}
\left|(q+F_{\bar n+k})\setminus F_{\bar n+k}\right|\leq \left|(c_{\bar n}+F_{\bar n+k})\setminus F_{\bar n+k}\right|=\frac{\left |F_{\bar n+k}\right|c_{\bar n}}{c_{\bar n+k}},
\end{equation*}
and so
\begin{equation*}
\lim_{n\to\infty}\frac{\left|q+F_{n}\setminus F_{n}\right|}{\left|F_n\right|}=\lim_{k\to\infty}\frac{\left|q+F_{\bar n+k}\setminus F_{\bar n+k}\right|}{\left|F_{\bar n+k}\right|}\leq \lim_{k\to\infty}\frac{c_{\bar n}}{c_{\bar n+k}}=0.
\end{equation*}
If we consider instead a negative $q\in \Q$, we can proceed in a similar way. We find $\bar n\in\N$ as before but using $-q$ and then we have
\begin{equation*}
(F_{\bar n+k}-c_{\bar n})\setminus F_{\bar n+k}=\left(\left\langle\frac{1}{a_{\bar n+k}}\right\rangle\cap[0,c_{\bar n})\right)-c_{\bar n}.
\end{equation*}
As before we conclude that
\begin{equation*}
\left|(q+F_{\bar n+k})\setminus F_{\bar n+k}\right|\leq \left|(F_{\bar n+k}-c_{\bar n})\setminus F_{\bar n+k}\right|=\frac{\left |F_{\bar n+k}\right|c_{\bar n}}{c_{\bar n+k}},
\end{equation*}
and then also in this case
\begin{equation*}
\lim_{n\to\infty}\frac{\left|(q+F_{n})\setminus F_{n}\right|}{\left|F_n\right|}\leq\lim_{k\to\infty}\frac{c_{\bar n}}{c_{\bar n+k}}=0.
\end{equation*}
Therefore $(F_n)_{n\in\N}$ is a F\o lner sequence of $\Q$.
\end{proof}

One can verify that the above F\o lner sequence of $\Q$ is also locally monotileable (see \cite{Flavio}). This is written in details in the next proof in a particular case (but in a more general setting), where one should take $q=1$.

\begin{proof}[\bf Proof of Theorem~\ref{tipro3}]
Each automorphism $\phi$ of the group $(\Q,+)$ is of the form $\phi_q\colon x\mapsto qx$ for some $q\in\Q\setminus\{0\}$. 
Fix $q=a/b\in\Q\setminus\{0\}$, with $(a,b)=1$, $a\in\Z\setminus\{0\}$ and $b\in\N\setminus\{0\}$.
Let $(E_n)_{n\in\N}$ with, for every $n\in\N$, 
$$E_n=\left\langle \frac{1}{|a|^nb^n(n!)}\right\rangle\cap[0,2^n|a|^nb^n).$$
First $(E_n)_{n\in\N}$ is a F\o lner sequence of $\Q$, since $(|a|^nb^n(n!))_{n\in\N}$ is a geometric a-sequence and $(2^n|a|^nb^n)_{n\in\N}$ is an a-sequence, and so Claim~\ref{monoqc} applies.
	
Now we verify that $(E_n)_{n\in\N}$ is $\phi_q$-monotileable, To this end, six $n\in\N$. If $a=|a|>0$, then
$$\phi_q(E_n)=qE_n=\left\langle \frac{1}{b^{n+1}|a|^{n-1}(n!)}\right\rangle\cap[0,2^n|a|^{n+1}b^{n-1}).$$
We note that
\begin{equation}\label{exeq1}
E_{n+1}=\bigsqcup_{j=0}^{2b^2}2^n|a|^{n+1}b^{n-1}j+\left(E_{n+1}\cap[0,2^n|a|^{n+1}b^{n-1})\right).
\end{equation}
Moreover,
\begin{equation}\label{exeq2}
E_{n+1}\cap[0,2^n|a|^{n+1}b^{n-1})=\bigsqcup_{i=0}^{|a|^2(n+1)-1}\frac{i}{|a|^{n+1}b^{n+1}(n+1)!}+\phi_q(E_n)
\end{equation}
Combining \eqref{exeq1} and \eqref{exeq2}, we obtain
\begin{equation}\label{exeq3}
E_{n+1}=\bigsqcup_{j=0}^{2b^2-1}\biggl( 2^n|a|^{n+1}b^{n-1}j+\biggl(\bigsqcup_{i=0}^{|a|^2(n+1)-1}\frac{i}{|a|^{n+1}b^{n+1}(n+1)!}+\phi_q(E_n)\biggr)\biggr).
\end{equation}
Define
\[\bar E^+_{n+1}=\biggl\{2^n|a|^{n+1}b^{n-1}j+\frac{i}{|a|^{n+1}b^{n+1}(n+1)!}: 0\leq i<|a|^2(n+1),\ 0\leq j<2b^2\biggr\}.\]
By Lemma~\ref{disun} and \eqref{exeq3},
\begin{equation}\label{exeq5}
E_{n+1}=\bigsqcup_{\bar e\in\bar E^+_{n+1}}\bar e+\phi_q(E_n),
\end{equation}
i.e., $\phi_q(E_n)$ is a monotile of $E_{n+1}$. Exactly in the same way we obtain that $\phi_q^{-1}(E_n)$ is a monotile of $E_{n+1}$.
	
If $a<0$, then 
$$\phi_q(E_n)=qE_n=\left\langle \frac{1}{|a|^{n-1}b^{n+1}(n!)}\right\rangle\cap(-2^n|a|^{n+1}b^{n-1},0]\quad\text{and}\quad
\phi_{-q}(E_n)=2^n|a|^{n+1}b^{n-1}-\frac{1}{|a|^{n-1}b^{n+1}(n!)}+\phi_q(E_n).$$
Since $a<0$, clearly $-q>0$, and so \eqref{exeq5} yields
\begin{equation}\label{exeq6}
E_{n+1}=\bigsqcup_{\bar e\in\bar E^+_{n+1}}\bar e+\phi_{-q}(E_n)=\bigsqcup_{\bar e\in\bar E^+_{n+1}}\left(\bar e+2^n|a|^{n+1}b^{n-1}-\frac{1}{|a|^{n-1}b^{n+1}(n!)}+\phi_q(E_n)\right).
\end{equation}
Define
$$\bar E^-_{n+1}=2^n|a|^{n+1}b^{n-1}-\frac{1}{|a|^{n-1}b^{n+1}(n!)}+\bar E^+_{n+1}.$$
By \eqref{exeq6} we have
$E_{n+1}=\bigsqcup_{\bar e\in\bar E^-_{n+1}}\bar e+\phi_q(E_n),$
and so $\phi_q(E_n)$ is a monotile of $E_{n+1}$. In the same way we find that $\phi_q^{-1}(E_n)$ is a monotile of $E_{n+1}$.
Hence, $(E_n)_{n\in\N}$ is $\phi_q$-monotileable. 

To conclude, apply Corollary~\ref{exco1}.
\end{proof}

Given a finitely generated subgroup $K$ of $\Aut(\Q)$, one can prove that $\Q \rtimes K$ is \tileable. Suppose that $K$ is generated by $\{\frac{a_1}{b_1},\ldots,\frac{a_m}{b_m}\}$. Consider the sequence $(E_n)_{n\in\N}$ given by
$$E_n=\left\langle \frac{1}{(|a_1\ldots a_m|b_1\ldots b_m)^n(n!)}\right\rangle\cap[0,(2^n|a_1\ldots a_m|^n(b_1\ldots b_m)^n).$$
Proceeding as in the proof of Theorem~\ref{tipro3} one could prove that $(E_n)_{n\in\N}$ is $K$-monotileable and then apply Theorem~\ref{expro1}. 

%
%

\subsection{Examples of locally monotileable groups that are not virtually nilpotent}\label{ex-sec}

Here we show that the general Extension Theorem can be proved for extensions
\begin{equation}\label{Christmas}
0 \rightarrow H \xrightarrow{\iota} G \xrightarrow{\pi} K \rightarrow 0
\end{equation}
when $H$ has a property a bit stronger than local finiteness. 

\begin{corollary}\label{new:corol}
Suppose that in the short exact sequence \eqref{Christmas} the group $H$ has the property that every finite subset of $H$ is contained in a finite characteristic subgroup of $H$. If $K$ is \tileable, then $G$ is \tileable\ as well. 
\end{corollary}
\begin{proof}
Clearly, $H$ is locally finite, hence \tileable\ by Proposition~\ref{monotor}. Moreover, there exists an exhausting increasing sequence $(E_n)_{n\in\N}$ of finite characteristic subgroups of $H$. In particular, $E_ n$ is invariant under conjugations in $G$. So $(E_n)_{n\in\N}$ is an  $\Aut(G)$-monotileable right F\o lner sequence of $G$. Therefore Claim~\ref{monosequ} gives that $G$ is \tileable.
\end{proof}

This corollary provides some new examples  of \tileable\  groups. 
 
\begin{example}
Take an infinite collection $\{S_n\colon n\in\N\}$ of simple finite groups such that for $n\ne m$ the only homomorphism $S_n \to S_m$ is the trivial one and $\exp(S_n)$ is not bounded. Let $H=\bigoplus_{n\in\N}S_n$ be as in Example~\ref{iii}(a). Then $H$ satisfies the hypotheses of Corollary~\ref{new:corol}. In particular, $H$ is locally finite, hence $H$ is \tileable\ by Proposition~\ref{monotor}. 

Let us see that for every residually finite countable abelian group $K$ one can define an appropriate faithful action $\theta$ of $K$ on $H$, such that  $H\rtimes_\theta K$ is \tileable. 
Indeed, $K$ can be identified with a subgroup of $P = \prod_{m\in\N_+}\Z(m)$.
Since $\Aut(H) = \prod_{n\in\N}\Aut(S_n)$, each $\Aut(S_n)$ contains a copy of $S_n$, hence the orders of these groups are not bounded.
Therefore, the product $\Aut(H)$ contains a subgroup isomorphic to $P$.
In particular, $P$ (hence, $\Aut(H)$ as well) contains an isomorphic copy of the group $K$, which gives rise to a faithful action $\theta$ of $K$ on $H$. 
Since the group $K$ is \tileable\ by Corollary~\ref{abelian}, we deduce that $G = H\rtimes_\theta K$ is \tileable by Corollary~\ref{new:corol}. 

For example, we can just pick a non-torsion element $\phi \in P$ and put $K = \langle \phi \rangle \cong \Z$; then $G = H\rtimes K\in \M$. 
Note that $G$ is neither locally finite, nor virtually solvable, nor residually solvable. 
\end{example}

\begin{example} 
We build here examples of \tileable\ amenable groups that are neither virtually nilpotent nor residually finite. 
\begin{enumerate}[(a)]
\item Suppose that in the short exact sequence \eqref{Christmas} the group $H$ is abelian and satisfies the descending chain condition on subgroups. If $K$ is \tileable, then $G$ is \tileable\ as well. 
Indeed, it is well known that the hypothesis on $H$ implies that $H\cong F \oplus \bigoplus_{i=1}^k\Z(p^\infty_{i})$, where $p_1, \ldots, p_k$ are not necessarily distinct primes (see \cite{Fuchs}). Then for every $n \in \N$ the subgroup $H[n]= \{h\in H\colon nh = 0\}$ is finite and fully invariant. Since every finite subset of $H$ is contained in some of the subgroups $H[n]$, the group $H$ satisfies the hypotheses of Corollary~\ref{new:corol}. Hence, $G$ is \tileable.

\item Take as in Example~\ref{iii}(b) $H =\bigoplus_{i=1}^k\Z(p^\infty_{i})$, where $p_1, \ldots, p_k$ are distinct primes. Then 
$K^*= \Aut(H) \cong \prod_{i=1}^kU(\mathbb J_{p_i})$,
where $U(\mathbb J_{p_i})$ is the group of units of the ring $\mathbb J_{p_i}$, and in particular $K^*= \Aut(H)$ is abelian. So any countable subgroup $K$ of $K^*$ is \tileable. 
Now consider the semidirect product $G = H \rtimes K$, where the action of $K$ is that induced by the natural one of $\Aut(H)$ on $H$. By item (a) $G$ is \tileable. 

\item Now take for simplicity $k=1$ and $p=p_1>2$ in item (b), so $H=\Z(p^\infty)$ and $K^* = U(\mathbb J_p)$. Since $p>2$, we can choose $K \not \subseteq 1 + p\mathbb J_p$, and so the natural action of $K$ on $H$ is fixed-point free. Therefore, $G = H \rtimes K$ is center-free, so $G$ is \tileable\ and $G$ is not nilpotent. 

To see that $G$ is not virtually nilpotent consider a finite index subgroup $N$ of $G$ contained in $K_1 = (1+ p^k \mathbb J_p)\cap K$. It is enough to show that $N$ is not nilpotent. Since $H$ is divisible, we deduce that $H\leq N$. Hence, $N = H \rtimes (N\cap K)$ and $m= [K:(N\cap K)]< \infty$, so $N \subseteq mK$. In particular, if $m=p^lm_1$, with $l\in \N$ and  $(m_1,p) = 1$, then there exists $ \xi \in N \cap (1+p^l\mathbb J_p)$. This implies that $Z(N) = Z_1(N) = \Z(p^l) \times \{1\}$. 
Arguing by induction one can see that $Z_s(N) = Z_1(N) = \Z(p^{ls}) \times \{1\}$. Hence, $N \ne Z_s(N)$ for every $s\in \N$. Thus $N$ is not nilpotent. Since $H = \bigcup _{s\in \N}  Z_s(N)$, $N$ is hypercentral. Therefore, $G$ is virtually hypercentral, yet not virtually nilpotent.  

Finally, it is easy to see that $G$ is not residually finite, as $H$ is a non-trivial divisible subgroup of $G$.
\end{enumerate}
\end{example}

\thebibliography{10}

{\footnotesize

\bibitem{AKM} R. Adler, A. Konheim, M. McAndrew, \emph{Topological entropy}, Trans. Amer. Math. Soc. 114 (1965) 309--319.

\bibitem{BG} M. Bertelson, M. Gromov, \emph{Dynamical Morse entropy. Modern dynamical systems and applications}, 27--44, Cambridge Univ. Press, Cambridge, 2004. 

\bibitem{B} R. Bowen, \emph{Entropy for group endomorphisms and homogeneous spaces}, Trans. Amer. Math. Soc. 153 (1971) 401--414.

\bibitem{CC} P. Cecchi, M. I. Cortez, \emph{Invariant measures for actions of congruent monotileable amenable groups}, Groups Geom. Dyn. (2019) doi:10.4171/GGD/506.

\bibitem{CCK} T. Ceccherini-Silberstein, M. Coornaert, F. Krieger, \emph{An analogue of Fekete's lemma for subadditive functions on cancellative amenable semigroups}, J. Anal. Math. 124 (2014) 59--81.


\bibitem{Con} J. P. Conze, \emph{Entropie d'un groupe ab\'{e}lien de transformations}, Z. Wahrscheinlichkeitstheorie und verwandte Gebiete 25 (1972) 11--30.

\bibitem{CP} M. I. Cortez, S. Petite, \emph{Invariant measures and orbit equivalence for generalized Toeplitz subshifts}, Groups Geom. Dyn. 8 (2014) 1007--1045.

\bibitem{CT} N. Chung, A. Thom, \emph{Some remarks on the entropy for algebraic actions of amenable groups}, Trans. Amer. Math. Soc. 367 (2015) 8579--8595. 

\bibitem{Dan} A. Danilenko, {\em Amenable groups are finitileable. Dynamical proof}, arXiv:1503.01568v1.



\bibitem{DFG-amac}  D. Dikranjan, A. Fornasiero, A. Giordano Bruno, \emph{The algebraic entropy of amenable semigroup actions}, submitted.
\bibitem{DFG} D. Dikranjan, A. Fornasiero, A. Giordano Bruno, \emph{Entropy of generalized shifts and related topics}, preprint.

\bibitem{DGB} D. Dikranjan, A. Giordano Bruno, \emph{The Pinsker subgroup of an algebraic flow}, J. Pure Appl. Algebra (2012) 364--376.
\bibitem{DGBpak} D. Dikranjan, A. Giordano Bruno, \emph{Topological and algebraic entropy on groups}, Proceedings Islamabad ICTA 2011, Cambridge Scientific Publishers 2012, 133--214.
\bibitem{DGB1} D. Dikranjan, A. Giordano Bruno, \emph{The connection between topological and algebraic entropy}, Topology Appl. 159 (2012) 2980--2989.
\bibitem{DGB2} D. Dikranjan, A. Giordano Bruno, \emph{The Bridge Theorem for totally disconnected LCA groups}, Topology Appl. 169 (2014) 21--32.
\bibitem{DGB0} D. Dikranjan, A. Giordano Bruno, \emph{Entropy on abelian groups}, Adv. Math. 298 (2016) 612--653.
\bibitem{DGSZ} D. Dikranjan, B. Goldsmith, L. Salce, P. Zanardo, \emph{Algebraic entropy for abelian groups}, Trans. Amer. Math. Soc. 361 (2009) 3401--3434.

\bibitem{Dik+Manolo} D. Dikranjan, M. Sanchis, \emph{Bowen's entropy for endomorphisms of totally bounded abelian Groups}, Descriptive Topology and Functional Analysis, Springer Proceedings in Mathematics \& Statistics, Volume 80 (2014) 143--162.

\bibitem{Din}  E. Dinaburg, \emph{On the relations among various entropy characteristics of dynamical systems}, Izv. Akad. Nauk SSSR 35 (1971) 324--366. 


\bibitem{DH} T. Downarowicz, D. Huczek,  {\em Dynamical quasitilings of amenable groups}, Bull. Pol. Acad. Sci. Math. 66 (2018) 45--55.
\bibitem{DHZ} T. Downarowicz, D. Huczek, G. Zhang, \emph{Tilings of amenable groups}, J. Reine Angew. Math. 747 (2019) 277--298.

\bibitem{Ebli} S. Ebli, \emph{On coarse tilings for actions of discrete groups}, Universit\`a di Padova, Master Thesis, 2016.


\bibitem{Fuchs} L. Fuchs, \emph{Abelian groups}, Springer Monographs in Mathematics, 2015.

\bibitem{GB} A. Giordano Bruno, \emph{A Bridge Theorem for the entropy of semigroup actions}, submitted.
\bibitem{GBSp} A. Giordano Bruno, P. Spiga, \emph{Some properties of the growth and of the algebraic entropy of group endomorphisms}, J. Group Theory 20 (4) (2017) 763--774.
\bibitem{GBSp1} A. Giordano Bruno, P. Spiga, \emph{Milnor-Wolf Theorem for group endomorphisms}, J. Algebra 546 (2020) 85--118.
\bibitem{GBSal} A. Giordano Bruno, F. Salizzoni, \emph{Additivity of the algebraic entropy for locally finite groups with permutable finite subgroups}, submitted.



\bibitem{KW} Y. Katznelson, B. Weiss, \emph{Commuting measure-preserving transformations}, Israel J. Math. 12 (1972) 161--173.
\bibitem{KL} D. Kerr, H. Li, \emph{Erogidic Theory, Independence and dichotomies}, Springer Monographs in Mathematics, Springer, Cham, 2016.

\bibitem{Li} H. Li, \emph{Compact group automorphisms, addition formulas and Fuglede-Kadison determinants}, Ann. of Math. (2) 176 (2012) 303--347. 
\bibitem{LL} H. Li, B. Liang, \emph{Sofic mean length}, Adv. Math. 353 (2019) 802--858. 

\bibitem{LSW} D. Lind, K. Schmidt, T. Ward, \emph{Mahler measure and entropy for commuting automorphisms of compact groups}, Invent. Math. 101 (1990) 593--629. 
\bibitem{N} I. Namioka, \emph{F\o lner's conditions for amenable semigroups}, Math. Scand. 15 (1964) 18--28.

\bibitem{Oll} J.M. Ollagnier, \emph{Ergodic theory and statistical mechanics}, Lecture Notes in Math. 1115, Springer Verlag, Berlin-Heidelberg-New York, 1985.

\bibitem{OW0} D. Ornstein, B. Weiss, \emph{Ergodic theory of amenable group actions I: The Rokhlin lemma}, Bull. Amer. Math. Soc. (N.S.) 2 (1980) 161--164.

\bibitem{OW} D. Ornstein, B. Weiss, \emph{Entropy and isomorphism theorems for actions of amenable groups}, J. Analyse Math. 48 (1987) 1--141.

\bibitem{P1} J. Peters, \emph{Entropy on discrete abelian groups}, Adv. Math. 33 (1979) 1--13.
\bibitem{P2} J. Peters, \emph{Entropy of automorphisms on LCA groups}, Pacific J. Math. 96 (1981) 475--488.

\bibitem{SVV} L. Salce, P. V\'amos, S. Virili, \emph{Length functions, multiplicities and algebraic entropy}, Forum Math. 25 (2013) 255--282.
\bibitem{SZ1} L. Salce, P. Zanardo, \emph{A general notion of algebraic entropy and the rank-entropy}, Forum Math. 21 (2009) 579--599.
\bibitem{SV} L. Salce, S. Virili, \emph{The addition theorem for algebraic entropies induced by non-discrete length functions}, Forum Math. 28 (2016) 1143--1157. 

\bibitem{Flavio} F. Salizzoni, \emph{Algebraic entropy of amenable semigroup actions}, Master Thesis, University of Udine, 2019.



\bibitem{V1} S. Virili, \emph{Entropy for endomorphisms of LCA groups}, Topology Appl. 159 (2012) 2546--2556.
\bibitem{V3} S. Virili, \emph{Algebraic entropy of amenable group actions}, Math. Z. 291 (2019) 1389--1417.
\bibitem{V2} S. Virili, \emph{Algebraic and topological entropy of group actions}, preprint.

\bibitem{Weiss} B. Weiss, \emph{Monotileable amenable groups}, Topology, ergodic theory, real algebraic geometry, 257--262, Amer. Math. Soc. Transl. Ser. 2, 202, Adv. Math. Sci., 50, Amer. Math. Soc., Providence, RI, 2001.


\bibitem{W} M. D. Weiss, \emph{Algebraic and other entropies of group endomorphisms}, Math. Systems Theory 8 (3) (1974/75) 243--248.

\bibitem{Y} S. Yuzvinski, \emph{Metric properties of endomorphisms of compact groups}, Izv. Acad. Nauk SSSR, Ser. Mat. 29 (1965) 1295--1328.

}

\end{document}